\documentclass[a4paper]{article}

\usepackage[english]{babel}

\usepackage[utf8]{inputenc}
\usepackage[utf8]{inputenc}
\usepackage{amsmath}
\usepackage{graphicx}
\usepackage{amssymb}
\usepackage{amsthm}
\usepackage{tikz-cd}
\usepackage{mathrsfs}
\usepackage[colorinlistoftodos]{todonotes}
\usepackage{enumitem}
\usepackage{yfonts}
\usepackage{ dsfont }
\usepackage{MnSymbol}
\usepackage{slashed}

\title{Intermediate complex structure limit for Calabi-Yau metrics}

\author{Yang Li}

\date{\today}
\newtheorem{thm}{Theorem}[section]
\newtheorem{lem}[thm]{Lemma}

\theoremstyle{definition}
\newtheorem{eg}[thm]{Example}

\newtheorem{cor}[thm]{Corollary}

\newtheorem{rmk}[thm]{Remark}
\newtheorem{prop}[thm]{Proposition}

\newtheorem*{Notation}{Notation}

\newtheorem*{Acknowledgement}{Acknowledgement}

\newcommand{\ie}{\emph{i.e.} }
\newcommand{\cf}{\emph{cf.} }

\newcommand{\R}{\mathbb{R}}
\newcommand{\C}{\mathbb{C}}
\newcommand{\Z}{\mathbb{Z}}
\newcommand{\N}{\mathbb{N}}
\newcommand{\Q}{\mathbb{Q}}

\newcommand{\norm}[1]{\left\lVert#1\right\rVert}

\def\Xint#1{\mathchoice
	{\XXint\displaystyle\textstyle{#1}}%
	{\XXint\textstyle\scriptstyle{#1}}%
	{\XXint\scriptstyle\scriptscriptstyle{#1}}%
	{\XXint\scriptscriptstyle\scriptscriptstyle{#1}}%
	\!\int}
\def\XXint#1#2#3{{\setbox0=\hbox{$#1{#2#3}{\int}$ }
		\vcenter{\hbox{$#2#3$ }}\kern-.6\wd0}}

\def\dashint{\Xint-}

\begin{document}
	\maketitle

\begin{abstract}
We study polarised degenerations of $n$-dimensional Calabi-Yau hypersurfaces $\{ F_0F_1\ldots F_m+tF=0\} $, where the essential skeleton has dimension $1\leq m\leq n-1$. We will describe the limiting behaviour of the Calabi-Yau potential at the $C^0$-level, in terms of an optimal transport problem, and in terms of non-archimedean potential theory.
\end{abstract}

\section{Introduction}

A \emph{polarised degeneration} of Calabi-Yau (CY) manifolds refers to a 1-parameter family of $n$-dimensional projective CY manifolds $X\to \mathbb{D}^*_t$ over a small complex disc, equipped with an ample line bundle $L\to X$. One can then consider the CY metrics $(X_t, \omega_{CY,t}, \Omega_t)$ in the class $c_1(L)$, and ask for the limiting behaviour as $t\to 0$.

This problem depends crucially on an integer invariant $0\leq m\leq n$, defined as the dimension of the essential skeleton $Sk(X)$, which controls the growth rate of the volume integral $\int_{X_t} \Omega_t\wedge \overline{\Omega}_t\sim O(|\log |t||^m)$ (see section \ref{sectionessentialskeleton}).
Some cases have recently received special attention:

\begin{itemize}
\item  In the $m=0$ case, the CY metrics $(X_t, \omega_t)$ are uniformly volume non-collapsing, and converge in the Gromov-Hausdorff sense to a singular CY metric on a CY variety with klt singularity \cite{DonaldsonSun}\cite{EGZ}\cite{RongZhang}. 

\item The opposite extreme $m=n$ is known as the \emph{large complex structure limit}, which is the context for the SYZ conjecture \cite{SYZ}. Recent progress \cite{LiNA} has indicated an intimate link with non-archimedean geometry.

\item The case $m=1$ (known as the `\emph{small complex structure limit}') is studied by Sun-Zhang \cite{SunZhang}, whose gluing construction covers many families of CY hypersurfaces inside Fano manifolds, degenerating into the union of two transverse smooth divisors. The special case for K3 surfaces was previously obtained in \cite{HSVZ}.

\end{itemize}

We focus on the case $1\leq m\leq n-1$, which we call the `\emph{intermediate complex structure limit}'. Let $M$ be an $(n+1)$-dimensional smooth Fano manifold, and we write its anticanonical bundle as $-K_M=(\sum_0^m d_i)L$
for some ample line bundle $L\to M$, and natural numbers $d_i>0$. Let $F_0, F_1,\ldots, F_m$ (resp. $F$) be sufficiently generic sections of $H^0(M, L^{\otimes d_i})$ (resp. $H^0(M, -K_M)$, so that all the divisors $E_i=(F_i=0)$ and $(F=0)$ are smooth, and all intersections are transverse. We consider the family of CY hypersurfaces $X\to \mathbb{D}_t^*$:
\begin{equation}\label{Fanohypersurface}
X_t= \{  F_0F_1\ldots F_m +tF=0   \} \subset M.
\end{equation}
Under the assumption $1\leq m\leq n-1$, the intersection $\cap_0^m E_i$ is connected using the Lefschetz hyperplane theorem, so the essential skeleton $Sk(X)$ consists of only one simplex of dimension $m$.

Our goal is to give two $C^0$ pluripotential theoretic limiting descriptions for the Hermitian metrics $(L, h_{CY,t})$ whose Chern curvature induce the CY metrics $(X_t, \omega_{CY,t})$ in the class $c_1(L)$. 

\begin{itemize}
    \item  The limit can be encoded by the solution of a real Monge-Amp\`ere type equation with $m$ independent variables, which arises from the unique solution of an optimal transport problem. This perspective has a more PDE theoretic flavour, which in the $m=1$ special case reduces to the limiting potential found by Sun-Zhang \cite{SunZhang}.

    \item The limit can be interpreted in terms of non-archimedean (NA) geometry on the Berkovich space $X_K^{an}$ associated to the degeneration family. NA pluripotential theory \cite{Boucksom} provides a unique up to scale semipositive metric $\norm{\cdot}_{CY,0}$ on the line bundle $L\to X_K^{an}$, whose NA Monge-Amp\`ere measure equals the Lebesgue measure supported on the essential skeleton $Sk(X)$ with total integral $(L^n)$. There is a hybrid topology on $X_K^{an}\cup \bigcup_{t\neq 0} X_t$ unifying the Berkovich space with the CY manifolds. Our main theorem is a continuity principle on the hybrid topology:

    \begin{thm}\label{mainthm}
    Let $1\leq m\leq n-1$. After suitable normalisation, the norm function $|\cdot |_{h_{CY,t}}^{\frac{1}{|\log |t||}}$ on $L\to X_t$ converges in the $C^0$-hybrid topology to $\norm{\cdot}_{CY,0}$ on $L\to X_K^{an}$ as $t\to 0$.

\end{thm}
    
The inspiration comes from the large complex structure case $m=n$, which is relevant to the SYZ conjecture \cite{LiNA}\cite{PilleSchneider}\cite{Hultgren}\cite{LiFano}\cite{Hultgren2}. Our case $1\leq m\leq n-1$ has some notable new features, ultimately due to the fact that non $\Q$-factorial dlt models do not specify a natural retraction map of the Berkovich space.
\end{itemize}

The organisation is as follows. Section \ref{NAbackground} recalls some basic notions in NA pluripotential theory, notably how the Berkovich space and the complex manifolds are unified through the hybrid topology. Section \ref{NACYFanohypersurface} derives an a priori structure result on the NA analogue of the CY metric on $L\to X_K^{an}$, via a NA version of Bergman kernel estimates and some algebro-geometric inputs. 
 Section  \ref{RealMAoptimaltransport} takes this structure result as a guide to formulate an optimal transport problem, whose unique minimizer produces a convex function solving a real MA type equation, with a curious wall-chamber structure corresponding to a simplicial decomposition of the gradient image. Section \ref{Ansatzmetrics} uses this optimal transport solution to construct positive ansatz metrics in both the archimedean and NA settings. The NA ansatz turns out to be equal to the NA CY metric, while the archimedean ansatz will be approximately CY in a rather weak sense, and the two ansatz constructions fit together  continuously in the hybrid topology. With more inputs from complex pluripotential theory, section \ref{ConvergenceCYpotential} proves that  after normalising by the factor $\frac{1}{|\log |t||}$, the difference between the archimedean ansatz and the CY potential converges to zero in the $C^0$-norm, which then implies the main theorem. Some further discussions are collected in section \ref{Furtherdiscussions}, notably the relation to Sun-Zhang's work \cite{SunZhang}, and some open questions.

\begin{Notation}
Our convention is $d=\partial+\bar{\partial}, d^c= \frac{\sqrt{-1}}{2\pi} (-\partial+ \bar{\partial})$, so $dd^c=\frac{\sqrt{-1}}{2\pi}\partial \bar{\partial}$. The relation between K\"ahler potentials and K\"ahler metrics is $\omega_\phi=\omega+dd^c\phi$.  Given a Hermitian metric $h$ on a line bundle $L$, its curvature form is $-dd^c\log h^{1/2}$ in the class $c_1(L)$.
\end{Notation}

\begin{Acknowledgement}
The author is a current Clay Research Fellow based at MIT. He thanks S. Sun, R. Zhang, S. Boucksom, J. Hultgren and L. Pille-Schneider for past discussions.
\end{Acknowledgement}

\section{Some background on NA geometry}\label{NAbackground}

\subsection{Models and dual complex}

We recall a few basic notions from non-archimedean geometry \cite[section 5]{Boucksom1}. Denote  $R=\C[\![ t ]\!]$, and $K=\C(\!(t)\!)$. Let $X_K$ be a smooth, connected, projective variety $X_K$ of dimension $n$ over the formal disc $\text{Spec}(K)$. Let $X_K^{an}$ denote the  \emph{Berkovich space}. Given an affine variety over $K$, the Berkovich space is the space of semivaluations on the ring of functions extending the standard discrete valuation on $K$, and in general $X_K^{an}$ is obtained by gluing the affine pieces.	The semivaluations $v$ are equivalently thought of as multiplicative seminorms $|\cdot |=e^{-v}$, satisfying the ultrametric property
\[
|f+g|\leq \max\{ |f|+|g|  \}, \quad |fg|=|f||g|.
\]

 A \emph{model} of $X_K$ is a normal separated scheme $\mathcal{X}$, flat and of finite type over $\text{Spec} (R)  $, together with an identification of the generic fibre of the morphism $\mathcal{X}\to\text{Spec} (R)  $ with $X_K$.  For any model $\mathcal{X}$ and every irreducible component $E$ of $\mathcal{X}_0$, we write the central fibre as $\mathcal{X}_0=\sum b_iE_i$, and $E_J= \cap_{i\in J}E_i$. We say $\mathcal{X}$ is SNC if $\mathcal{X}$ is regular and $\mathcal{X}_0$ has SNC support. In general, every model contains a largest SNC Zariski open subset $\mathcal{X}_{SNC}\subset \mathcal{X}$, and by Hironaka's theorem we can always resolve a given model to achieve an SNC model. A model of the line bundle $L\to X_K$ is a $\Q$-line bundle $\mathcal{L}$ on a proper model $\mathcal{X}$, together with an identification $\mathcal{L}|_{X_K}=L$.

Let $\mathcal{X}$ be a model with log canonical divisor $K_{\mathcal{X}}^{log}$ being $\Q$-Cartier. We say the model $\mathcal{X}$ is \emph{divisorially log terminal} (dlt) if the pair $(\mathcal{X}, \mathcal{X}_{0,red})$ is dlt. This means that $(\mathcal{X}, \mathcal{X}_{0,red})$ is a log canonical pair, and all lc centres are contained in the SNC locus $\mathcal{X}_{SNC}$.

\begin{rmk}
The motivation for dlt models is partly due to the minimal model programme \cite{NicaiseXu}\cite{NicaiseXuYu}, and partly because SNC models typically involve a huge number of divisors, which complicates the combinatorics. The differential geometric intuition for dlt models is that they behave like SNC models, modulo some Zariski closed locus whose contribution to the volume integrals is negligible.
\end{rmk}

The \emph{dual complex} $\Delta_{\mathcal{X}}$ of an SNC model $\mathcal{X}$ is the simplicial complex whose vertices correspond to the irreducible components $E_i$ of the central fibre, and whose faces correspond to the connected components of nonempty intersection strata $E_J$, and are identified as the simplex 
$\Delta_J\simeq \{  x\in \R_{\geq 0}^{|J|}| \sum_{i\in J} b_ix_i=1 \}.
$
The dual complex of a dlt model is defined as the dual complex of $\mathcal{X}_{SNC}$.

Given a dlt model,
there is a natural \emph{embedding map} of the dual complex into $X_K^{an}$, which is a homeomorphism onto its image. Given $x\in \Delta_J$, we need to associate a quasi-monomial valuation $v_x$, regarded as a point in $X_K^{an}$. Given any local function $f\in \mathcal{O}_{\mathcal{X},\eta}$ where $\eta$ is the generic point of $E_J$, we can Taylor expand
\[
f=\sum_{\alpha\in \N^{|J|}} f_\alpha z_1^{\alpha_0}\ldots z_{|J|}^{\alpha_{|J|}},
\] 
where $z_i$ are local equations of the divisors $E_i$, and $f_\alpha$ is either a unit or zero. Then $v_x$ is defined as the minimal weighted vanishing order:
\[
v_x(f)= \min\{   \langle x, \alpha\rangle  |  f_\alpha\neq 0  \}.
\]
In particular, the vertices of $\Delta_J$ map to the divisorial valuations. The requirement $\sum b_ix_i=1$ comes from $v(t)=1$, where $t\in R$ is the uniformizer. Henceforth we identify the dual complex with its image in $X_K^{an}$.

\begin{rmk}
Any SNC model induces a \emph{retraction map} $r_{\mathcal{X}}: X_K^{an}\to \Delta_{\mathcal{X}}$, via computing the valuation on the local defining functions of $E_i$. This retraction map makes sense for $\Q$-factorial dlt models $\mathcal{X}$, but its existence is unclear for general dlt models, since $E_i$ may not be Cartier divisors, hence do not necessarily have local equations. We warn the reader that the retraction map must be treated with extra care, and otherwise dlt models work almost the same way as SNC models. We will not use retraction maps in this paper.
\end{rmk}

\begin{eg}\label{Fanohypersurfaceeg1}
In the example (\ref{Fanohypersurface}), the preferred model over $\text{Spec}(R)$
\[
\mathcal{X}= \{ F_0F_1\ldots F_d +tF=0  \}\subset M_R
\]
is \emph{dlt but not $\Q$-factorial}. This is because in our setup, the local singularity of $\mathcal{X}$ is transversely modelled on $z_1\ldots z_k=ty$, where $z_i$ are the local defining equations for the components of the toric boundary, and $y$ is the local coordinate for  $F$ (\cf \cite[Prop. 4.1]{PilleSchneider}). In particular, the multiplicities $b_i=1$ in this example.

\end{eg}

\subsection{Hybrid topology}

Given a degeneration family $X\to \mathbb{D}^*$ of smooth projective manifolds $X_t$ over some punctured Riemann surface, we can associate a Berkovich space $X_K^{an}$ to the base changed family $X_K$ over the formal disc. There is a natural topology, known as the \emph{hybrid topology}, which describes $X_K^{an}$ as a limiting object of the complex manifolds \cite[Appendix]{Boucksom} \cite[Section 1.2]{PilleSchneider}.

Abstractly, for any scheme $X$ of finite type over a Banach ring $A$ (not necessarily non-archimedean!), Berkovich associates an \emph{analytification} $X^{An}$. When $X=\text{Spec}(B)$ is affine, with $B$ a finitely generated $A$-algebra, then $X^{An}$ is defined as the set of multiplicative seminorms $\norm{\cdot}_x$ on $B$ whose restriction to $A$ is bounded by the given norm on $A$, and the topology on $X^{An}$ is the weak topology. The general case is obtained by gluing affine pieces.
The analytification $X^{An}$ has a natural map to the \emph{Berkovich spectrum} of $A$, which is the set of bounded multiplicative seminorms on $A$.

The hybrid norm on the complex field $\C$ is defined as
\[
\norm{\cdot}_{hyb}= \max\{  |\cdot|_0, |\cdot|_\infty   \},
\]
where $|\cdot|_0$ is the trivial absolute value, and $|\cdot|_\infty$ is the usual absolute value.

Let $A_{e^{-1}}$ be the Banach ring \footnote{One can replace $e^{-1}$ by $0<r<1$ via some scaling of the base $t$-variable.}
\[
A_{e^{-1}}=\{   f=\sum_{\alpha\in \Z} c_\alpha t^\alpha\in \C(\!(  t )\!)| \norm{f}_{hyb}:= \sum \norm{c_\alpha}_{hyb} e^{-\alpha}        \}.
\]
The Berkovich spectrum over $A_{e^{-1}}$ is known as the \emph{hybrid circle}, which is homeomorphically identified with the closed disc $\bar{\mathbb{D}}_{e^{-1}}=\{ t\in \C: |t|\leq e^{-1}  \}$ as follows: given $z\in \bar{\mathbb{D}}_{e^{-1}}$, we associate the seminorm on $A_{e^{-1}}$ by
\[
|f|=\begin{cases}
e^{ - ord_0(f) }, \quad & z=0,
\\
e^{ - \log |f(z)|_\infty /\log |z|_\infty   },\quad & \text{otherwise}.
\end{cases}
\]
Let $X$ now be a scheme of finite type over $A_{e^{-1}}$; the concrete source comes from a family of smooth projective manifolds over the punctured disc. The \emph{hybrid analytification} $X^{hyb}$ is the Berkovich analytification over the hybrid ring $A_{e^{-1}}$, which has a natural map $\pi$ to the Berkovich spectrum $\bar{\mathbb{D}}_{e^{-1}}$. By \cite[Lemma A.6]{Boucksom}, which uses the Gelfand-Mazur theorem, we have canonical identifications $\pi^{-1}(0)=X_K^{an}$, and $\pi^{-1}(\mathbb{D}_{e^{-1}}^*)=X^{hol}$, \ie the central fibre corresponds to the NA Berkovich space, and the other fibres recover the complex manifold fibres. 
Given a family of continuous functions $f_t$ on $X_t$, and a continuous function $f_0$ on $X_K^{an}$, we will say that $f_t$ converges in $C^0$-hybrid topology to $f_0$ as $t\to 0$, if whenever the points $z_t\in X_t$ converge to $z_0\in X_K^{an}$, we have $f_t(z_t)\to f_0(z_0)$.

\subsection{Essential skeleton and volume form asymptote}\label{sectionessentialskeleton}

Suppose $X$ is a degeneration family of smooth projective CY manifolds $X_t$ over some punctured disc, and $X$ is equipped with a holomorphic volume form $\Omega$, which induces the holomorphic volume forms $\Omega_t$ on the $X_t$ via $\Omega= \Omega_t\wedge dt$. The \emph{essential skeleton} is the subset of $X_K^{an}$ where the log discrepancy function takes the minimal value, which we may without loss assume to be zero, up to multiplying $\Omega$ by a suitable power of $t$.

We follow \cite{Boucksom} to sketch the computation of the CY volume form asymptote. Let $\mathcal{X}$ be a dlt model of the degeneration family, whose dual complex is $\Delta_{\mathcal{X}}$. We work on the SNC locus, and write the central fibre as $\mathcal{X}_0=\sum b_iE_i$, and the relative log canonical divisor as $\sum a_i E_i$ (so the log discrepancy function takes the value $\frac{a_i}{b_i}$ at the divisorial point corresponding to $E_i$). Pick a closed point $\xi\in \mathcal{X}_0$, and let $E_i$ be the divisor components of $\mathcal{X}$ passing through $\xi$, for $i\in J=\{ 0,1,\ldots k \}$. We may choose a local coordinate system $z_0,\ldots z_n$ around $\xi\in \mathcal{X}$, such that $z_0,\ldots z_k$ are the local defining equations for $E_0,\ldots E_k$. Locally
\[
\Omega=f_1 z_0^{a_0+b_0-1}\ldots z_k^{a_k+b_k-1} dz_0\wedge \ldots  dz_n,
\]
for some invertible local holomorphic function $f_1$, and the base variable is
\[
t= f_2 z_0^{b_0}\ldots z_k^{b_k} ,
\]
for some invertible local holomorphic function $f_2$. Then
\begin{equation}\label{holovollocal}
\Omega_t= b_0^{-1} f_1f_2^{-1} z_0^{a_0}\ldots z_k^{a_k} d\log z_1\wedge \ldots d\log z_k\wedge dz_{k+1}\wedge \ldots dz_n.
\end{equation}
By the non-negativity of the log discrepancy function, we have $a_i\geq 0$ for $i=0,\ldots k$.

If at least one exponent $a_i>0$, then the local contribution to the volume integral $\int_{X_t} \Omega_t\wedge \overline{\Omega}_t$ is suppressed by a factor $O(\frac{1}{|\log |t||})$. The divisors with $a_i=0$ correspond to the vertices of a subcomplex of $\Delta_{\mathcal{X}}$, which is naturally identified with the essential skeleton $Sk(X)\subset X_K^{an}$ under the embedding map.  The dimension $m=\dim Sk(X)$ then equals the largest value of $k$ such that $a_0=\ldots =a_k=0$, so the total volume integral $\int_{X_t} \Omega_t\wedge \overline{\Omega}_t=O(|\log |t||^m)$. We denote the normalized CY measure as
\[
\mu_t= \frac{  \Omega_t\wedge \overline{\Omega}_t }{   \int_{X_t} \Omega_t\wedge \overline{\Omega}_t      } ,
\]
which is a probability measure on $X_t$.
Boucksom-Jonsson \cite{Boucksom} show that

\begin{thm}\label{Volumeconvergence}
Under the hybrid topology convergence, the measures $\mu_t$ on $X_t$ converge weakly to a probability measure $\mu_0$ supported on $Sk(X)\subset X_K^{an}$, which puts no measure on the lower dimensional faces of $Sk(X)$, and agrees with a suitably normalized Lebesgue measure on the $m$-dimensional faces of $Sk(X)$.
More concretely, for any continuous function $f$ on the hybrid space $X^{An}$, we have
\[
\int_{X_t} fd\mu_t\to \int_{X_K^{an}} fd\mu_0.
\]
\end{thm}

\begin{rmk}
	The reason the volume convergence result works the same way for dlt models as for SNC models, is that the singular locus under further blow up will produce divisors with strictly positive value for the log discrepancy, hence its volume contribution is suppressed.

\end{rmk}

The normalisation constants can be pinned down by a more invariant interpretation of the local computation (\ref{holovollocal}). Let $E_J=\cap_0^m E_i$ correspond to an $m$-dimensional face $\Delta_J\subset Sk(X)$. Complex geometrically, the local structure of $\mathcal{X}$ around (the SNC locus of) $E_J$ is modelled on the total space of the vector bundle $\oplus_{i\in J} \mathcal{O}(E_i) \to E_J$. The formula (\ref{holovollocal}) can be written as
\[
\Omega_t= \Omega_{E_J}\wedge d\log z_1\wedge\ldots d\log z_k,
\]
where $\Omega_{E_J}$ extends holomorphically over $E_J$ (since all $a_i=0$ here), and defines a holomorphic volume form on $E_J$ called the Poincar\'e residue, which is independent of the choice of local coordinates. The integral
\[
\int_{E_J} \Omega_{E_J}\wedge \overline{\Omega}_{E_J}
\]
is finite, since the dlt condition guarantees  the convergence of the integral around the (possibly nonempty) singular locus of $E_J$. Then as a measure on the face $\Delta_J\simeq \{  x\in \R_{\geq 0}^{m+1}|\sum_0^m b_i x_i=1  \}$,
\begin{equation}\label{Lebesguemeasure}
\mu_0= (C_0 \int_{E_J} \Omega_{E_J}\wedge \overline{\Omega}_{E_J}) |dx_1\ldots dx_m|,
\end{equation}
where the constant $C_0$ is independent of the face, and only serves to normalise $\mu_0$ to a probability measure.

\begin{eg}\label{Fanohypersurfaceexample2}
We continue with the Fano hypersurface example (\ref{Fanohypersurface}). Algebro-geometrically, the hypersurfaces $X_t$ degenerates to the union of $\{  F_i=0\}$ as $t\to 0$, and the normalised volume measure $\mu_t$ concentrates near $E_J=\cap_0^{m+1}E_i$ up to $O(\frac{1}{|\log|t||})$ error. By the Lefschetz hyperplane theorem, $E_J$ is connected, so the essential skeleton consists of just one $m$-dimensional simplex.

The neighbourhood of $E_J$ inside the Fano manifold is complex analytically modelled on the total space of $\oplus_0^m L^{\otimes d_i}\to E_J$, and $X_t$ is cut out by $\{ F_0\ldots F_m=-tF \}$ with $F_i\in H^0( L^{\otimes d_i})$. On the part of $X_t$ in the $O(\frac{1}{\log |t||})$-neighbourhood near $E_J\subset M$, the holomorphic volume form is up to $O(\frac{1}{|\log |t||})$-error
\[
\Omega_t\approx \Omega_{E_J}\wedge d\log F_1\wedge \ldots d\log F_m.
\]
The normalisation convention for $\mu_0$ to be a probability measure requires
\begin{equation}\label{normalisationC0}
C_0 \int_{E_J}\Omega_{E_J}\wedge \overline{\Omega}_{E_J}= \frac{1}{\int_{Sk(X)}  dx_1\ldots dx_m }= m!,
\end{equation}
namely
$
d\mu_0= m! |dx_1\ldots dx_m|.
$
Correspondingly, up to $O(\frac{1}{|\log |t||})$ error,\footnote{The notation $\log |F_i|$ involves an implicit choice of local trivialisation of $L$, so that $F_i$ can be identified with a holomorphic function.} 
\begin{equation}\label{normalisedCYmeasure}
d\mu_t\approx C_0\Omega_{E_J}\wedge \overline{\Omega}_{E_J}\wedge \bigwedge_1^m \frac{1}{  |\log |t|| } d\log |F_i|\wedge d^c\log |F_i|.
\end{equation}

\end{eg}

\subsection{Semipositive metrics, NA Calabi conjecture}

A \emph{continuous metric} $\norm{\cdot}$ on the line bundle $L\to X_K^{an}$ associates to local sections $s$ a continuous function $\norm{s}$, compatible with the restriction of sections, and such that $\norm{fs}=|f|\norm{s}$ for local functions $f$. A choice of model line bundle $\mathcal{L}\to \mathcal{X}$ induces a `model metric' on $L$, with the property that the locally invertible sections have norm one. We can then write any other continuous metric in terms of a potential function, via $\norm{\cdot}=\norm{\cdot}_{model}e^{-\phi}$, and it is conventional to identify the metric with the potential.

Let $L$ be a semiample line bundle. Given $k\geq 1$, and a collection $(s_1,\ldots s_N)$ of global sections of $kL$ without common zeros, and real constants $c_1,\ldots c_N$, then we can associate the \emph{non-archimedean Fubini-Study metric} by
\begin{equation}
\norm{s}_{FS,k}(x)= \frac{ |s|(x) } { ( \max_{1\leq i\leq N}{ |s_i|(x) e^{c_i} ) }^{1/k}   }.
\end{equation}
This can be recast in terms of potentials
\begin{equation}\label{NAFubiniStudy}
\phi_{FS,k}= \frac{1}{k} \max_{1\leq i\leq N}(\log |s_i| +c_i),
\end{equation}
where to evaluate $|s_i|$, we implicitly use a local trivializing section of $L$ coming from the choice of a model line bundle.

A continuous \emph{semipositive metric} on $L$ can be then defined as a uniform limit of some sequence of NA Fubini-Study metrics. To such metrics we can associate the NA Monge-Amp\`ere measure $MA(\phi)$, with total integral $(L^n)$. For model line bundles, the NA MA measure is defined in terms of intersection theory, and puts delta masses on the vertices of the dual complex inside $X_K^{an}$, and the general case is uniquely defined such that the MA measures converge weakly under the uniform convergence of metrics.

The central result of NA pluripotential theory is the solution of the NA Calabi conjecture:

\begin{thm}
\cite{Boucksomsurvey}\cite{Boucksom} Let $L\to X$ be an ample line bundle. Given a probability measure $\mu$ supported on some dual complex $\Delta_{\mathcal{X}}\subset X_K^{an}$, then there exists a unique (up to an additive constant) continuous semipositive metric $\phi$ on $L$ such that $MA(\phi)=(L^n)\mu$.
\end{thm}

\begin{eg}\label{Fanohypersurfaceeg2}
The most important case is when $\mu$ is the Lebesgue measure $\mu_0$ on the essential skeleton (\cf (\ref{Lebesguemeasure})). The corresponding solution $\norm{\cdot}_{CY,0}$ is referred to as the \emph{non-archimedean Calabi-Yau (NA CY) metric}. In our application to Fano hypersurfaces (\cf Example \ref{Fanohypersurfaceeg1}), the line bundle $L$ on the Fano manifold $M$ provides a preferred model line bundle over the model $\mathcal{X}$. Suppose $N_0$ is large enough so that $H^0(M, N_0L)$ globally generate the line bundle $L^{\otimes N_0}\to M$, then upon picking a basis $\tau_0,\ldots \tau_{N_1}$ of $H^0(M, N_0L)$, the induced NA Fubini-Study metric agrees with the preferred model metric on $L\to X_K$, which is independent of the choice of $N_0$ or the basis. Using this model metric, the NA CY metric $\norm{\cdot}_{CY,0}$ is identified naturally with a potential $\phi_{CY,0}$.	
\end{eg}

\subsection{NA MA measure vs. complex MA}\label{NAMAcxMA}

The following `hybrid topology continuity principle' for the NA MA measure is essentially contained in the work of Favre \cite{Favre}, and further elucidated by Pille-Schneider \cite{PilleSchneider2}.


We fix some ambient projective embedding of the degeneration family $X$ with ample polarisation $L\to X$, and fix a smooth Hermitian metric $h$ on  $L$ with positive curvature over the ambient space. The choice of $h$ does not matter. Then the Fubini-Study metric (\ref{NAFubiniStudy}) can be viewed as a hybrid topology limit as follows. 
Given the data of the global sections $(s_1,\ldots s_N)$ of $kL$ with no common zeros, and the real constants $c_1,\ldots c_N$, We can then define the  continuous potentials on $(X_t,L)$,
\[
\phi_{k,t}= \frac{1}{k} \max_{1\leq i\leq N} (\log |s_i|_{h^{\otimes k}}+c_i |\log |t||),
\]
corresponding to the Hermitian metrics $h_{k,t}= he^{-2\phi_{k,t}}$ on $(X_t, L)$, which has non-negative curvature current on $X_t$. Then
\begin{itemize}
\item The potential $\frac{1}{|\log |t||}\phi_{k,t}\to \phi_{FS,k}$ as $t\to 0$ in the $C^0$-hybrid topology. Equivalently, the archimedean norm functions $|s|_{h_{k,t}}^{1/|\log |t||}\to \norm{s}_{FS,k}$ in $C^0$-hybrid topology for any local algebraic sections $s$ of $L$. 

\item  The complex MA measure on $X_t$ associated to $h_{k,t}$ converges weakly as $t\to 0$ to the NA MA measure on $X_K^{an}$. (This is based on interpreting the concentration of complex MA measures in terms of intersection numbers, see \cite[Prop. 3.1]{Favre}.)
\end{itemize}

The more general potentials can be handled by uniform Fubini-Study approximation. Let $\phi_t$ be continuous $\omega_h$-psh potentials on $(X_t, L)$ satisfying the uniform approximation property: there exists small $r_k>0$, and a sequence of approximants $\phi_{k,t}$ as above with $k\to +\infty$, such that
\begin{equation}\label{uniformapproximation1}
\sup_{X_t} | \phi_{k,t}-\phi_t| \leq \epsilon_k |\log |t||,\quad \forall 0<|t|\leq r_k.
\end{equation}
and $\epsilon_k\to 0$ as $k\to +\infty$. Then there is a continuous NA psh potential $\phi$ on $(X_K^{an}, L)$ which arises as the $C^0$-limit of $\frac{1}{|\log |t||}\phi_t$ as $t\to 0$. Morever, the complex MA measures on $X_t$ associated to $\phi_t$ converge as the NA MA measure of $\phi$ as $t\to 0$. (This is based on the Chern-Levine inequality, see \cite[section 4]{Favre}.)

\begin{rmk}
Favre \cite[section 4]{Favre} required the small radii $r_k>0$ to be independent of $k$. This is not necessary if we a priori assume $\phi_t$ to have the $\omega_h$-psh property. Notice we do not require $\phi_t$ to have psh dependence on the family parameter $t$.
\end{rmk}

\section{NA CY metric for the Fano hypersurfaces}\label{NACYFanohypersurface}

We focus on the family of CY hypersurfaces $X$ as in (\ref{Fanohypersurface}), base changed to a family $X_K\to \text{Spec}(K)$ inside the Fano manifold $M_K$ over $\text{Spec}(K)$, and study the NA CY metric $\phi_{CY,0}$. The main result of this section is Prop. \ref{NACYpotentialstructure}, which expresses $\phi_{CY,0}$ as a convex function with gradient constraints, with the arguments being $\log |F_0|,\ldots \log |F_m|$. Along the way we will discover a natural decomposition for the gradient image of the convex function, which will eventually be reflected in a wall-chamber structure for the potential theoretic limit of the CY metrics. 

\subsection{Decomposition of sections}\label{Decompositionofsections}

We start with some preparations in algebraic geometry, with the goal of understanding the NA Fubini-Study metrics.

We consider any nonzero section $s\in H^0(X_K, lL)$ for $l\geq 1$. From the exact sequence of sheaves on $M_K$,
\[
0\to (l-\sum d_i)L\to lL\to  \mathcal{O}_{X_K}(lL)\to 0,
\]
and the Kodaira vanishing 
\[
H^1(M_K, (l-\sum d_i)L) \simeq H^1(M_K, lL+ K_M  )=0,
\]
we deduce that the restriction map $H^0(M_K, lL)\to H^0(X_K, lL)$ is surjective, so the section $s$ can be lifted to $H^0(M_K, lL)$. Since $X_K$ is connected, we can include the $l=0$ case.

We write $E_J=\{  F_i=0, \forall i \}\subset X$. By the same kind of exact sequence argument, the restriction
\[
H^0(M, kL)\to H^0(E_J, kL), \quad k\geq 0,
\]
is surjective, so we can pick a subspace $V_k\subset H^0(M,kL)$ mapping isomorphically to $H^0(E_J, kL)$.

\begin{lem}
The natural map of $\C$-vector spaces \[
W_l=\bigoplus_{\sum d_il_i\leq l}  F_0^{l_0}F_1^{l_1}\ldots F_m^{l_m} V_{l-\sum d_i l_i}\to H^0(M, lL)
\] is an isomorphism.
\end{lem}

\begin{proof}
The map is injective, because the distinct summands have different vanishing orders along the divisors $E_i=(F_i=0)$, so that given a decomposition $\sum F_0^{l_0}\ldots F_m^{l_m} s_{l_0,\ldots l_m}=0$, we can inductively prove the summands are zero.

To prove surjectivity, it suffices to compare the dimension, so it suffices to prove the equality of the generating functions
\[
\sum_{l\geq 0} \dim W_l  y^l= \sum_{l\geq 0} h^0(M,lL) y^l.
\]
The RHS is the Hilbert polynomial $P_M(y)$ for $(M,L)$. The LHS is 
\[
(1- y^{d_0})\ldots (1-y^{d_m}) P_{E_J}(y),
\]
where $P_{E_J}(y)$ is the Hilbert polynomial of $(E_J,L)$. The result then follows from the well known relation of Hilbert polynomials
\[
P_{E_J}(y)=\frac{1}{ (1- y^{d_0})\ldots (1-y^{d_m})} P_M(y).
\]
\end{proof}

We thus identify $W_l$ with $H^0(M,lL)$, and
denote
$\tilde{W}_l\subset W_l$ as the subspace of $W_l$,
consisting of those indices $(l_0,\ldots l_m)$ with at least one zero entry.

\begin{lem}
The section $s\in H^0(X_K, lL)$ can be lifted to $\tilde{W}_l\otimes_\C K$.

\end{lem}

\begin{proof}
We first lift $s$ to $H^0(M_K, lL)$. After multiplying $s$ by a fixed finite power of $t$, we have $t^N s\in H^0(M, lL)\otimes R$.   Notice
\[
H^0(M, lL)=\tilde{W}_l \oplus F_0\ldots F_m H^0(M, (l-\sum d_i)L)
\]
Upon restriction to $X$, we can replace $F_0\ldots F_m $ with $-tF$, so that
\[
t^N s\in \text{Image}( \tilde{W}_l\otimes R + t H^0(M,lL)\otimes R).
\]	
Iterating the argument,
\[
t^N s\in  \text{Image}(  \tilde{W}_l\otimes R + t^{N'} H^0(M,lL)\otimes R  ), \quad \forall N'.
\]
Passing to the completion and dividing by $t^N$ yields the claim.	
\end{proof}

Thus $s\in H^0(X_K, lL)$ has an expansion
\begin{equation}\label{monomialdecomposition}
s=\sum_{l_0,\ldots l_m}  F_0^{l_0}\ldots F_m^{l_m} \sum_a s_{l_0\ldots l_m, a} t^a,
\end{equation}
where the integer $a$ is bounded from below, $l_0,\ldots l_m$ are non-negative integers with at least one zero entry, and $s_{l_0\ldots l_m, a}\in V_{l-\sum d_il_i}$ is either zero, or does not vanish identically along the intersection stratum $E_J$.

By picking any local trivialising section of $L^{\otimes l}$, we can regard $t^N s$ as a local function in $k(E_J)[[z_0,\ldots z_m]]$, where $z_i$ are the local defining equations for $E_i=(F_i=0)$.  We can then evaluate $s$ on the quasi-monomial valuation points inside the essential skeleton $Sk(X)\subset X_K^{an}$. We recall that $Sk(X)$ is identified with the simplex
\begin{equation}
Sk(X)\simeq \Delta= \{  \sum_0^m x_i=1, x_i\geq 0     \}\subset \R^{n+1},
\end{equation}
since the dlt model has only one depth $m$ intersection stratum, and all divisor multiplicities $b_i=1$.

\begin{cor}\label{logsonessentialskeleton}
	(No cancellation)
As a function on $Sk(X)$, 
\[
-\log | s| =\min \{   l_0x_0+ \ldots l_m x_m + a   : s_{l_0\ldots l_m,a}\neq 0     \}.
\]
By construction $d_0l_0+\ldots d_ml_m\leq l$. 
\end{cor}

\begin{proof}
Recall that to compute the quasi-monomial valuation on a function, we need to Taylor expand in the $z_0,\ldots z_m$ variables. Notice that by $F_0\ldots F_m=-tF$, and the fact that $F$ is nonzero in $k(E_J)$, we know $t$ can be expanded as the product $z_0\ldots z_m$ (up to a unit in $k(E_J)$, which is inessential).

The crucial point is that all the exponents in the monomial decomposition (\ref{monomialdecomposition}) are \emph{distinct}, so there is \emph{no cancellation effect}. Concretely, given
\[(l_0+a,\ldots , l_m+a)= (l_0'+a',\ldots l_m'+a'),
\]
we need to check $l_0=l_0',\ldots l_m=l_m', a=a'$. Here it is essential that at least one of $l_0,\ldots l_m$ is zero, so that $a$ can be recovered as $\min(l_0+a,\ldots l_m+a)$, which shows $a=a'$, and the claim is now clear.
\end{proof}

\subsection{Fubini-Study approximation}

We now consider the NA CY metric $\phi_{CY,0}$ (\cf Example \ref{Fanohypersurfaceeg2}). By the characterization of continuous semipositive metrics as uniform limits of Fubini-Study metrics, we can find a sequence of Fubini-Study metrics of the form
\[
\phi_j=\frac{1}{l} \max_{i=1}^{N} ( \log |s_i|+ c_i), \quad s_i\in H^0(X, lL), \quad c_i\in \R,
\]
such that \[
-\epsilon_j\leq \phi_j- \phi_{CY,0}\leq 0,\quad \epsilon_j\to 0.
\]
Here we suppress the dependence of $l, N$ etc on the sequential index $j$, and in particular $l$ is unbounded.

We expand as in (\ref{monomialdecomposition}),
\[
s_i= \sum F_0^{l_0}\ldots F_m^{l_m}\sum s^{(i)}_{l_0\ldots l_m,a}t^a.
\]
By Cor. \ref{logsonessentialskeleton}, restricted to the essential skeleton, each $-\log |s_i|$ is of the form 
\[
-\log |s_i|=\min\{ l_0x_0+\ldots l_mx_m +a |  s^{(i)}_{l_0\ldots l_m,a}\neq 0 \}.
\]
Here $d_0l_0+ \ldots d_ml_m\leq l$, and at least one of $l_i$ is zero by construction. We can then introduce the new potential
\begin{equation}\label{improvedFS}
\tilde{\phi}_j= \frac{1}{l} \max_{i=1}^N \max\{  l_0\log |F_0|+\ldots +l_m\log |F_m| +c_i-a : s^{(i)}_{l_0\ldots l_m,a}\neq 0    \}.
\end{equation}
Notice that only finitely many terms actually contribute to the maximum, since the $a\gg 1$ terms are very negative.
We list some properties:

\begin{lem}
The potential $\tilde{\phi}_j$ is a NA Fubini-Study metric, which agrees with $\phi_j$ on $Sk(X)$, and $\tilde{\phi}_j\geq \phi_j$ on $X_K^{an}$. 
\end{lem}

\begin{proof}
The equality $\phi_j=\tilde{\phi}_j$ holds on $Sk(X)$ by construction. To see that $\tilde{\phi}_j$ is a Fubini-Study metric, we recall from Example \ref{Fanohypersurfaceeg2} that the identification of metrics with potentials involves a preferred model metric, which is itself a NA Fubini-Study metric
$
\frac{1}{N_0}\max_p \log |\tau_p|.
$
Then $\tilde{\phi}_j$ is a Fubini-Study metric of the form
\[
\frac{1}{l}  \max\{  (l_0\log |F_0|+\ldots +l_m\log |F_m|+\frac{(l-\sum d_kl_k)}{N_0}\max \log |\tau_p|)+c_i-a : s^{(i)}_{l_0\ldots l_m,a}\neq 0    \}.
\]
This way of writing makes it manifest that the total degree of the sections add up to $l$. It is also clear that the common zero locus of all the sections involved in $\tilde{\phi}_j$ must be contained in the analogous locus for $\phi_j$, which is empty by assumption.

Finally we need to show $\tilde{\phi}_j\geq \phi_j$ on $X_K^{an}$. This follows from the ultrametric property
\[
|s_i|\leq \max \{  |F_0|^{l_0}\ldots |F_m|^{l_m} |s_{l_0\ldots l_m,a}| |t|^a        \},
\]
and 
\[
\log |s_{l_0\ldots l_m,a}| \leq \frac{l-\sum d_kl_k }{N_0}\max\{  \log |\tau_p|  \},
\]
which hold on all semi-valuations, not just those on $Sk(X)$.
\end{proof}

\begin{lem}\label{domination}
The potential $\tilde{\phi}_j\leq \phi_{CY,0}$. 
\end{lem}

\begin{proof}
Both potentials correspond to continuous semipositive metrics, and $\tilde{\phi}_j=\phi_j\leq \phi_{CY,0}$ holds on  $\text{supp} MA(\phi_{CY,0})=Sk(X)$. By the domination principle \cite[Lemma 8.4]{Boucksom}, we obtain $\tilde{\phi}_j\leq \phi_{CY,0}$ on $X_K^{an}$.
\end{proof}

As a Corollary, $-\epsilon_j\leq \tilde{\phi}_j-\phi_{CY,0}\leq 0$. 
The upshot is that we can replace $\phi_j$ by an improved Fubini-Study metric $\tilde{\phi}_j$, where instead of using arbitrary sections, we are using only specific sections coming from monomial expressions of $F_0,\ldots F_m$.

\subsection{Legendre transform}\label{Legendretransform}

We now reinterpret (\ref{improvedFS}) via the Legendre transform. We introduce simplices $\Delta^\vee_k$ for $k=0,1,\ldots m$, and denote $\Delta^\vee=\cup_k \Delta^\vee_k$:
\[
\Delta^\vee_k= \{  (p_0,\ldots p_m)\in \R^{m+1}_{\leq 0} | \sum d_ip_i \geq -1, p_k=0         \}.
\]
Suppose $u$ is a function on $\R^{m+1}$, of the form
\begin{equation}\label{Legendre1}
u(x)= \sup_{ p\in \Delta^\vee } \langle x, p\rangle - \tilde{u}(p), 
\end{equation}
for some bounded function $\tilde{u}$ on $\Delta^\vee$. Clearly $u$ is convex, and the gradient is contained in the convex hull of $\Delta^\vee$, so the Lipschitz constant of $u$ is a priori  bounded.

We can define a new function on $\Delta^\vee$ by
\begin{equation}
u^*(p)= \max_{x\in \Delta} (\langle x,p\rangle- u(x)), \quad \Delta=\{ x\in \R^{m+1}_{\geq 0}| \sum_0^m x_i=1  \},
\end{equation}
The Lipschitz constant of $u^*$ is a priori bounded. The double Legendre transform is defined as
\[
u^{**}(x)= \max_{p\in \Delta^\vee} \langle p,x\rangle- u^*(p).
\]

\begin{lem}
On $\Delta$ we have $u^{**}(x)=u(x)$. Globally on $\R^{m+1}$, we have $u^{**}\geq u$.
\end{lem}

\begin{proof}
From $\langle x, p\rangle - u(x)\leq \tilde{u}(p)$, clearly we have $u^*(p)\leq \tilde{u}(p)$. Thus
\[
u^{**}(x)= \max_{p\in \Delta^\vee} \langle p,x\rangle- u^*(p)\geq  \max_{p\in \Delta^\vee} \langle p,x\rangle- \tilde{u}(p) = u(x). 
\]
For the converse direction, notice
$
u^*(p)\geq \langle p,x\rangle -u(x)
$ for $x\in \Delta$,
so \[
u^{**}(x)= \max_{p\in \Delta^\vee} \langle p,x\rangle- u^*(p)\leq  u(x),\quad x\in \Delta.
\]
\end{proof}

\begin{rmk}\label{Legendredualsubdivision}
Under the quotient map $\R^{m+1}\to \R^{m+1}/\R(1,1,\ldots 1)$, the simplicial complex $\Delta^\vee$ maps homeomorphically onto the $m$-dimensional simplex $\bar{\Delta}^\vee$ with vertices $(0,\ldots -\frac{1}{d_i},\ldots 0)$:
\begin{equation}
\bar{\Delta}^\vee= \text{conv}( (0,\ldots -\frac{1}{d_i},\ldots 0) , i=0,1,\ldots m) \subset \R^{m+1}/\R(1,1,\ldots 1).
\end{equation}
The simplicial complex structure induces a subdivision of the image simplex:
\begin{equation}
\bar{\Delta}^\vee=\cup_k \bar{\Delta}^\vee_k,\quad \bar{\Delta}^\vee_k= \text{Image}(\Delta^\vee_k).
\end{equation}
 When we restrict to $\Delta\subset \R^{m+1}$, the formula (\ref{Legendre1}) is sensitive to the $p$ variable only modulo $\R(1,\ldots 1)$, and  amounts to requiring that $u$ is convex and the gradient of $u$ is contained in the simplex $\bar{\Delta}^\vee$.
\end{rmk}

\begin{prop}\label{NACYpotentialstructure}
The CY potential $\phi_{CY,0}$ can be expressed as $u(-\log |F_0|,\ldots -\log |F_m|)$ for some $u$  which  satisfies $u=u^{**}$ on $\R^{m+1}$.
\end{prop}

\begin{proof}
We return to the Fubini-Study metric $\tilde{\phi}_j$ from formula (\ref{improvedFS}). We notice $(-\frac{l_0}{l},\ldots, -\frac{l_m}{l})$  naturally lies in $\Delta^\vee$, since $\sum d_il_i\leq l$ and at least one of $l_i$ is zero. 
Thus (\ref{improvedFS}) can be rewritten as
\[
\tilde{\phi}_j= u_j( -\log |F_0|,\ldots -\log|F_m|   ),
\]
for some $u_j$ in the form (\ref{Legendre1}).

If we replace $u_j$ by $u_j^{**}$, then the new Fubini-Study potential $u_j^{**}(-\log |F_0|,\ldots)$ agrees with $\tilde{\phi}_j$ on the essential skeleton, because $u_j=u_j^{**}$ on $\Delta$, and globally on $X_K^{an}$ the potential can only increase due to $u_j^{**}\geq u_j$. Morever, by the same domination principle as in Lemma \ref{domination}, the new potential is still bounded above by $\phi_{CY,0}$. At this stage the potential shares all the properties required for $\tilde{\phi}_j$, so without loss $u_j=u_j^{**}$ over $\R^{m+1}$. 
In particular, the functions $u_j^*, u_j$ have uniform a priori Lipschitz estimates. By Arzela-Ascoli, we can extract a subsequential $C^0$-limit, \[
u_j^*\to u^*,\quad u_j\to u.
\]
Since $-\epsilon_j\leq \tilde{\phi}_j-\phi_{CY,0}\leq 0$ for $\epsilon_j\to 0$, by taking the limit, we deduce $\phi_{CY,0}$ has the required form.
\end{proof}

\begin{rmk}
	In the argument above, the only feature of the NA CY metric we need is that its NA MA measure is supported on $Sk(X)$.
\end{rmk}

\begin{rmk}\label{Vilsmeieranalogue}
In the analogous case of large complex structure limit $m=n$, C. Vilsmeier \cite{Vilsmeier} proved an explicit formula relating the NA MA measure to the real MA measure, under the assumption that the NA potential has an invariance property under the retraction map from $X_K^{an}$ to the dual complex of some semistable SNC model. It is an interesting problem to prove an analogous formula when $m<n$, under the assumption that the NA potential can be expressed as $u(-\log |F_0|,\ldots -\log |F_m|)$ for some $u$ of the form (\ref{improvedFS}), but not necessarily assuming the CY condition. It seems plausible that the NA MA measure will be closely related to the real MA type measure discussed in section \ref{realMAtypemeasure}, but the lack of a natural retraction map associated to dlt models complicates the issue.

\end{rmk}


\section{Real Monge-Amp\`ere type equation and Optimal transport problem}\label{RealMAoptimaltransport}

We now take a different perspective, and construct the candidate potential theoretic limit by solving a (nontrivial) real MA type equation through optimal transport. The motivation for  prescribing the real MA type equation comes from the generalised Calabi ansatz in \cite{LiNA}\cite{CollinsLi} (see also the volume computation in section \ref{Ansatzmetrics} for a posteriori justification). A curious feature of the equation is that the gradient image has a natural simplicial decomposition, and the locus where the gradient of the solution hits the intersections of these simplicial cells, behaves like a free boundary.

\subsection{Real MA type measure}\label{realMAtypemeasure}

We now introduce a real MA type measure, with some weighting factor depending on the gradient of the convex function.

Recall $\Delta=\{ \sum x_i=1, x_i\geq 0   \}\subset \R^{m+1}$.
Let $u$ be a convex function on $\Delta$ of the form (\ref{Legendre1}), namely the gradients will be contained in $\bar{\Delta}^\vee$ (\cf Remark \ref{Legendredualsubdivision}). At each interior point $x\in \text{Int}(\Delta)$, we denote the set of subgradients 
\[
\nabla u(x)= \{  p\in \Delta^\vee\simeq  \bar{\Delta}^\vee |  u(x')\geq u(x)+ \langle x-x', p\rangle \text{ for all } x'\in \Delta  \}
\]
 Since  $\langle x-x', (1,1,\ldots 1)\rangle =0$, the identification of $\Delta^\vee$ with its image $\bar{\Delta}^\vee$ in $\R^{m+1}/\R(1,\ldots 1)$ (\cf Remark \ref{Legendredualsubdivision}) is compatible with the definition of subgradients.

On each simplex $\Delta^\vee_j\subset \Delta^\vee$, we introduce a weighting factor 
\[
W(p)= (1+\sum_0^{m} d_ip_i)^{n-m},  
\]
which defines a Lipschitz continuous function on $\Delta^\vee\simeq \bar{\Delta}^\vee$. The simplices $\Delta^\vee_j$ carry the Lebesgue measure $dp=|dp_0\ldots dp_{j-1}dp_{j+1}\ldots dp_m|$, which via the identification of $\Delta^\vee$ with its image in $\R^{m+1}/\R(1,\ldots 1)$ agrees with the Lebesgue measure on the image simplex. We then define the set function 
\[
Mu(E)=\int_{\nabla u(E) } W dp ,\quad E\subset \text{Int}(\Delta).
\]
This is analogous to the weak definition of the real MA measure. Notice that we have the a priori bound $Mu(\text{Int}(\Delta))\leq \int_{\Delta^\vee}Wdp.$ The usual real MA measure theory \cite[Thm. 1.1.13]{Gutierez} works almost verbatim to show

\begin{lem}
The set function $Mu$ defines a finite Borel measure on $\text{Int}(\Delta)$, absolutely continuous with respect to the real Monge-Amp\`ere measure of $u$.
\end{lem}



We will be interested in solving the real MA type equation for a convex function $u$ as above,
\begin{equation}
Mu = C_1 \mu_0, \quad C_1= \int_{ \Delta^\vee } Wdp= \frac{d_0+\ldots +d_m}{d_0\ldots d_m} \frac{(n-m)!}{n!}  .
\end{equation}
where $\mu_0$ is the Lebesgue measure on $\Delta$ normalised to be a probability measure.

\begin{rmk}
Since the numbers $d_i$ may not be equal, the equation is generally not symmetric under the permutation group. Rescaling the coordinates $x_i$ do not restore the symmetry, since $\sum x_i$ will then be replaced by a weighted sum.
\end{rmk}

\subsection{Optimal transport}

The optimal transport problem is the following: given two probability measures $\mu$ and $\nu$ on (bounded) measurable subsets $A, B\subset \R^m$, find a measurable map $T: A\to B$ such that the pushforward of $\mu$ equals $\nu$, in such a way to minimize the transportation cost 
\[
\int_A c(x, T(x)) d\mu(x),\quad c(x,y)= \frac{1}{2} |x-y|^2.
\]
For a good survey, see \cite[section 3]{DePhilippisFigali}.

In our application, we choose $A=\Delta\subset \{\sum x_i=1\}\subset \R^{m+1}$, $B= \bar{\Delta}^\vee\subset \R^{m+1}/\R (1,\ldots 1)$, and the measures
\[
\mu= C_1\mu_0 \text{ on $\Delta$},\quad  \nu= W(p) dp \text{ on $\bar{\Delta}^\vee\simeq \Delta^\vee$}.
\]
Clearly $\mu, \nu$ are compactly supported, and absolutely continuous with respect to the Lebesgue measure. By Brenier's theorem \cite[Thm 3.1]{DePhilippisFigali}, there is a unique solution $T$ to the optimal transport problem. Morever, there exists a convex function $u: \Delta\subset \R^m\to \R$ such that the optimal map $T$ is given by $T(x)=\nabla u(x)$ for $\mu$-a.e. $x$, the map $T$ is differentiable $\mu$-a.e., and 
\begin{equation}\label{Brenier1}
|\det(\nabla T)| = \frac{C_1}{  W(T(x)) \int_{\Delta}dx  }, \quad \text{for $\mu$-a.e. $x\in \Delta$}.
\end{equation}
Since $\bar{\Delta}^\vee$ is convex, and $\Delta=\text{supp}(\mu)$, it follows that the optimal transport $T=\nabla u$ maps $\Delta$ into  $\bar{\Delta}^\vee$. Consequently $u$ is of the form (\ref{Legendre1}). Since $\nabla u$ is $\mu$-a.e. differentiable, (\ref{Brenier1}) implies that
\[
Mu(E) \geq \int_{ \nabla u (E\cap \text{Dom}(D^2 u) ) } W(p) dp = \int_{E\cap  \text{Dom}(D^2 u)} C_1d\mu_0 = \int_{E} C_1d\mu_0 ,\quad \forall E\subset \Delta.
\]
On the other hand, the total measure
\[
Mu(\Delta)=\int_{\nabla u(\Delta)} W(p)dp \leq \int_{ \Delta^\vee } W(p)dp = C_1.
\]
So all equalities must be achieved, whence we deduce

\begin{lem}
The solution to the optimal transport problem satisfies the real MA type equation $Mu=C_1\mu_0$, and $\nabla u(\Delta)=\bar{\Delta}^\vee$.
\end{lem}

\begin{eg}
The case of $m=1$ can be solved explicitly. The simplex $\Delta=\{ x_0+x_1=1,x_0\geq 0, x_1\geq 0  \}$ can be identified with the interval $[0,1]$ with variable $x=x_0$. The gradient simplices are
\[
\Delta^\vee_0= \{  (0, p_1): 0\geq p_1\geq -\frac{1}{d_1}   \}, \quad \Delta^\vee_1= \{  (p_0, 0): 0\geq p_0\geq -\frac{1}{d_0}   \},
\]	
so $\bar{\Delta}^\vee$ can be identified with the interval $[-\frac{1}{d_0}, \frac{1}{d_1}  ]$. The weighting factor is
\[
W(p)= \begin{cases}
(1+ d_0 p_0)^{n-1}, \quad \text{on } \Delta^\vee_1,\\
(1+ d_1 p_1)^{n-1}, \quad \text{on } \Delta^\vee_0.
\end{cases}
\]
Thus $\int_{\Delta^\vee} W(p)dp= \frac{1}{nd_0}+ \frac{1}{nd_1}$. The convex function is identified with $u=u(x)$, which satisfies 
\[
\begin{cases}
(1+ d_0 u')^{n-1} u''= \frac{1}{nd_0}+ \frac{1}{nd_1},\quad  \text{on } \Delta_1,
\\
(1- d_1 u')^{n-1} u''= \frac{1}{nd_0}+ \frac{1}{nd_1}, \quad \text{on } \Delta_0,
\end{cases}
\]
and the boundary constraint $u'(0)= - \frac{1}{d_0}, u'(1)=  \frac{1}{d_1}$. The solution is
\begin{equation}\label{optimaltransportm=1}
u'=
\begin{cases}
\frac{1}{d_0}\left(  -1+ (\frac{d_0+d_1}{d_1}x)^{1/n}   \right),\quad & 0\leq x\leq \frac{d_1}{d_0+d_1},
\\
\frac{1}{d_1}\left(  1- (\frac{d_0+d_1}{d_0}(1-x))^{1/n}   \right),\quad
& \frac{d_1}{d_0+d_1}\leq x\leq 1.
\end{cases}
\end{equation}
and $u=\int u'dx$. We see that $u$ is smooth on the two open intervals, while $u'$ is H\"older continuous at the interval boundary, and $C^{1,1}$ continuous at $x=\frac{d_1}{d_0+d_1}$.

\end{eg}

\subsection{Regularity of the solution}

We apply some standard regularity theory of real MA equation to the optimal transport solution. The results are far from optimal, but will suffice for our main purpose.

\begin{lem}
The subset 
\[
E=\{  x\in \Delta: \exists p\in \nabla u(x) \cap   \cup_k \partial \bar{\Delta}^\vee_k   \}
\]
is closed and has zero Lebesgue measure.
\end{lem}

\begin{proof}
The boundary of $\bar{\Delta}^\vee_k$ consists of a finite number of closed faces. As a general fact for convex functions, if $p_k\in \nabla u(x_k)$, and $x_k\to x, p_k\to p$, then $p\in \nabla u(x)$. Consequently, the gradient preimage is still a closed subset.

The subset $\mathcal{S}=\{ x\in\Delta: \nabla u(x) \text{ contains more than one element} \} $ has measure zero. By construction $\nabla u(E\setminus \mathcal{S})$ is contained in the measure zero subset  $\cup_k\partial\bar{\Delta}^\vee_k$, so the equation $Mu=C_1\mu_0$ implies that $E\setminus \mathcal{S}$ has zero Lebesgue measure, and so must $E$.
\end{proof}

This induces a \emph{wall-chamber structure} on $\Delta$:

\begin{cor}
The subset 
$\Delta_k=\{    x\in   \text{Int}(\Delta): \nabla u(x) \subset \text{Int}( \bar{\Delta}^\vee_k )    \}  $
 is an open subset, and $\cup_k \Delta_k$ has full Lebesgue measure in $\Delta$. 
\end{cor}

On each $\Delta_k$, the real MA type equation simplifies to
\begin{equation}\label{realMAtypeeqn}
\det(D^2 u) W(\nabla u)= \frac{C_1}{ \int_{\Delta}dx  }=  \frac{d_0+\ldots +d_m}{d_0\ldots d_m} \frac{(n-m)!m!}{n!} .
\end{equation}
By construction $0<W(\nabla u)\leq C$ on $\Delta_k$, and the function $W(p)$ is now analytic, instead of piecewise analytic.

\begin{lem}\label{uregularity}
The function $u$ is smooth on $\Delta_k$ after possibly deleting a closed subset of measure zero.
\end{lem}

\begin{proof}
First notice that $\det(D^2 u)\geq C^{-1}$, so by Mooney's partial regularity \cite[Thm 1.1]{Mooney}, $u$ is strictly convex away from a closed subset of Hausdorff $(n-1)$-measure zero. We claim that $u$ is smooth on the strictly convex locus.

By an exhaustion argument, it suffices to prove the smoothness on the subset with $W(\nabla u)>\epsilon$. By the two-sided density bound
\[
C^{-1}\leq \det(D^2 u)\leq C,
\]
we can apply Caffarelli's interior $C^{1,\alpha}$-estimate \cite[Thm 2.14]{DePhilippisFigali} to deduce $u\in C^{1,\alpha}_{loc}$. This implies $W(\nabla u)\in C^{\alpha}_{loc}$, which feeds back into the equation to show $u\in C^{2,\alpha}_{loc}$ by another result of Caffarelli \cite{CaffarelliC2alpha}. The rest follows from standard Schauder theory.
\end{proof}

\subsection{Extension of the convex function}

The convex function $u$ is a priori defined on $\Delta$. We now extend $u$ canonically to $\R^{m+1}$, via the double Legendre transform construction as in section \ref{Legendretransform}:
\[
u^*(p)= \max_{x\in \Delta} \langle x,p\rangle- u(x),  \quad p\in \Delta^\vee,
\]
\[
u^{**}(x)= \max_{p\in \Delta^\vee} \langle p, x\rangle- u^*(p).
\]
As in section \ref{Legendretransform}, the function $u^{**}$ agrees with $u$ on $\Delta$, and we will henceforth write $u=u^{**}$ as this  extension. Clearly $u$ is a convex function with gradient contained in the convex hull of $\Delta^\vee$. As a caveat, the distinction between $\Delta^\vee$ and $\bar{\Delta}^\vee$ is important for the extension, even though this previously makes no difference for $u$ restricted to $\Delta$.

\begin{lem}\label{extensiononUk}
There exists an open neighbourhood $U_k$ inside $\R^{m+1}$ of the chamber $\Delta_k\subset \text{Int}(\Delta)$, such that the extended convex function $u$ depends only on $x_0,\ldots x_{k-1}, x_{k+1},\ldots x_m$ on $U_k$.
\end{lem}

\begin{proof}
Given any compact subset $K\subset \Delta_k$, we can find some $\epsilon>0$ depending on $K$ such that for any $x\in K$,
\[
u(x) > \sup_{p\in \Delta^\vee\setminus \Delta^\vee_k } \langle p, x\rangle - u^*(p) +\epsilon.
\]
By the Lipschitz continuity of $u$, the same must hold with $K$ replaced by an open neighbourhood inside $\R^{m+1}$. This shows that on some open neighbourhood $U_k$, we have
\[
u(x)= \sup_{p\in  \Delta^\vee_k } \langle p, x\rangle - u^*(p).
\]
But $\Delta^\vee_k$ is contained in the plane $\{p_k=0 \}$, hence $u$ is independent of $x_k$ on $U_k$.
\end{proof}

\section{Ansatz metrics}\label{Ansatzmetrics}

Starting from the solution $u$ to the optimal transport problem,  our goal is to construct continuous psh potentials defining semipositive metrics in the class $c_1(L)$, in both the archimedean and the non-archimedean setting. The K\"ahler version has volume measure close to being CY in total variation norm, and the NA version turns out to be precisely the NA CY metric.

\subsection{Non-archimedean semipositive metric}\label{NAansatz}

As in Example \ref{Fanohypersurfaceeg2}, let $\tau_0,\ldots \tau_{N_1}$ be a basis of $H^0(M, N_0L)$ which globally generate the ample line bundle. The NA Fubini-Study metric on $(X_K^{an},L)$
\[
\frac{1}{N_0}\log \max |\tau_i| 
\] 
agrees with the preferred model metric on $L$. This reference metric allows us to identify other metrics on $L$ with potentials.

We define the continuous NA potential on $L$
\begin{equation}
\phi^{NA}= u( -\log |F_0|,\ldots -\log |F_m|    ),
\end{equation}
where $u$ is the solution to the optimal transport problem, and the $\log |F_i|$ makes sense using the preferred model metric on $L$.

\begin{prop}
The potential $\phi^{NA}$ defines a continuous semipositive metric. 
\end{prop}

\begin{proof}
We construct  $\phi_k^{NA}$ on $(X_K^{an}, L)$ by
\begin{equation}\label{phikNA}
\phi_k^{NA}=  \max  \{  -\sum_0^m p_i \log |F_i| -u^*(p)  | p=(p_0,\ldots p_m)\in  \Delta^\vee\cap \frac{1}{k}\Z^{m+1}      \}.
\end{equation}
We claim these are Fubini-Study potentials. To see this, notice the metric associated to $\phi_k^{NA}$ can be rewritten in homogeneous form as
\[
 \max  \{   \frac{1}{k}\log \prod_{i=0}^m |F_i|^{-kp_i} + \frac{1+ \sum  d_i p_i}{N_0} \max \log |\tau_l| -u^*(p)  | p\in  \Delta^\vee\cap \frac{1}{k}\Z^{m+1}      \}.
\]
Since $-p_i\geq 0$ and $1+\sum d_ip_i\geq 0$, this is the maximum for a collection of logarithmic monomial terms. Observe these monomials have no common zero locus.

Morever, the Lipschitz continuity of $u^*$ easily implies
\[
-Ck^{-1}\leq \phi_k^{NA} - \phi^{NA} \leq 0.
\]
We conclude that $\phi^{NA}$ is a uniform limit of NA Fubini-Study potentials, hence defines a continuous semipositive metric.
\end{proof}


\subsection{Ambient K\"ahler metrics}

Our  goal is to produce an ansatz K\"ahler metric on $X_t$ which is close to being CY except on a subset with small measure. The construction is guided by the a priori expectation that the CY potentials should $C^0$ converge to the NA CY potential.

On the $(n-m)$-dimensional compact CY manifold $E_J=\{ F_0=\ldots F_m=0 \}$, we find the CY metric $\omega_{E_J}$ in the class $c_1(L)$: 
\[
\omega_{E_J}^{n-m}= (L^{n-m}\cdot E_J) \frac{ \Omega_{E_J}\wedge \overline{\Omega}_{E_J} }{ \int_{E_J}\Omega_{E_J}\wedge \overline{\Omega}_{E_J}    } = d_0\ldots d_m (L^{n+1}\cdot M)\frac{ \Omega_{E_J}\wedge \overline{\Omega}_{E_J} }{ \int_{E_J}\Omega_{E_J}\wedge \overline{\Omega}_{E_J}   }, 
\]
where we recall $\Omega_{E_J}$ is the holomorphic volume form on $E_J$. 
Comparing with the normalisation convention (\ref{normalisationC0}), we see
\begin{equation}\label{omegaEJ}
\omega_{E_J}^{n-m}=  d_0\ldots d_m (L^{n+1}\cdot M)\frac{C_0}{m!}   \Omega_{E_J}\wedge \overline{\Omega}_{E_J} .
\end{equation}
We extend $\omega_{E_J}$ to a smooth K\"ahler metric $\omega_M$ on $M$ in the class $c_1(L)$, and denote $h$ as the corresponding Hermitian metric on $L\to M$. The choice of this extension does not matter since we will mostly be concerned only with the neighbourhood of $E_J$ inside $M$. 

We introduce the continuous potential on each $X_t$ for $0<|t|\ll 1$,
\begin{equation}
\phi_t=|\log |t|| u(- \frac{\log |F_0|_{h^{\otimes d_0}}}{|\log |t||},  \ldots, - \frac{\log |F_m|_{h^{\otimes d_m}}}{|\log |t||}   ),
\end{equation}
where $|F_i|_{h^{\otimes d_i}}$ denote the magnitude of the sections $F_i$ with respect to the Hermitian metrics on $L^{\otimes d_i}\to M$ induced by $h$.
This defines a $(1,1)$-current on $(X_t,  c_1(L))$ by
$
\omega_t=  \omega_M + dd^c \phi_t.
$

\begin{prop}
The potential $\phi_t$ is $\omega_M$-psh, namely the  $(1,1)$-current $\omega_t\geq 0$. 
\end{prop}

\begin{proof}
Since the $\omega_M$-psh property is stable under taking maximum, and
\[
\phi_t= \max_{p\in \Delta^\vee} \left( \sum_0^m -p_i \log |F_i|_{h^{\otimes d_i}} - u^*(p) |\log |t|| \right),
\]
it suffices to show that $\sum_0^m -p_i \log |F_i|_{h^{\otimes d_i}}$ is $\omega_M$-psh. Writing $h$ in terms of local potentials as $h=e^{-2\phi_h}$, then
\[
\begin{split}
& \omega_M+ dd^c \sum_0^m -p_i \log |F_i|_{h^{\otimes d_i}} = dd^c\phi_h + dd^c \sum_0^m -p_i \log (|F_i|e^{-d_i\phi_h}) 
\\
= &  (1+\sum d_ip_i) \omega_M -\sum p_i dd^c\log |F_i|\geq 0,
\end{split}
\]
where we used that $p_i\leq 0$, $1+\sum d_ip_i\geq 0$.
\end{proof}

We now verify the uniform approximation property in section \ref{NAMAcxMA}.  We produce the Fubini-Study type approximants closely related to the NA version in section \ref{NAansatz},
\[
\phi_{k,t}= \max\{   \frac{1}{k} \log \Pi_{i=0}^m |F_i|_{h^{\otimes d_i}}^{-kp_i} + \frac{1+\sum d_ip_i}{N_0} \max \log |\tau_l|_{h^{\otimes N_0}} -u^*(p)|\log |t||: p\in \Delta^\vee\cap \frac{1}{k} \Z^{m+1}            \}.
\]

\begin{prop}
(Uniform approximation) $|\phi_{k,t}-\phi_t|\leq C(1+ k^{-1}|\log |t||) $ for $0<|t|\ll 1$.
\end{prop}

\begin{proof}
The second term in $\phi_{k,t}$ is bounded by
\[
| \max_l \log |\tau_l|_{h^{\otimes N_0}} |\leq C.
\]
After deleting this term, the difference between $\phi_{k,t}$ and $\phi_t$ is caused by the discreteness condition $p\in \frac{1}{k}\Z^{m+1}$. By the Lipschitz bound on $u^*(p)$, this error is bounded by $|\log |t||Ck^{-1}$.
\end{proof}

In particular (\ref{uniformapproximation1}) holds upon choosing $r_k\ll e^{-k}$. By the discussion of section \ref{NAMAcxMA}, the $\phi_k^{NA}$ in (\ref{phikNA}) is the $C^0$-limit of $\frac{1}{|\log |t||}\phi_{k,t}$ in the hybrid topology. The uniform approximation then implies

\begin{cor}\label{hybridlimit1}
The NA potential $\phi^{NA}$ is the $C^0$ limit of $\frac{1}{|\log |t||}\phi_t$ in the hybrid topology. Morever, the NA MA measure associated to $\phi^{NA}$ is the weak limit of the complex MA measures $\omega_t^n$ as $t\to 0$.
\end{cor}

\subsection{Volume measure of the K\"ahler ansatz}

The main advantage of the K\"ahler ansatz $\omega_t$ is that the volume form is approximately CY in the generic region.

\begin{prop}\label{totalvariation1}
Given any $0<\delta\ll 1$, then for any sufficiently small $t$ depending on $\delta$,  there exists an open subset $U_{t,\delta}$ on $X_t$ with the normalised CY measure $\mu_t(U_{t,\delta})>1-\delta$, where the volume form is $C^0$-close to being CY:
\[
|{\omega_t^n} - (L^n) \mu_t| \leq \delta \mu_t.
\]
\end{prop}

\begin{proof}
By discarding a subset of $X_t$ with measure $\ll \delta$, we can focus on the neighbourhood of $E_J$ inside $X_t\subset M$, where $M$ is locally modelled on a holomorphic vector bundle over $E_J$. We consider the subsets with
\[
x=( -\frac{\log |F_0|_{h^{\otimes d_0}} }{|\log |t||}, \ldots  -\frac{\log |F_m|_{h^{\otimes d_m} }}{|\log |t||} )\in U_k, \quad k=0, 1, \ldots m,
\]
where $U_k$ is an open neighbourhood of $\Delta_k$ as in Lemma \ref{extensiononUk}. Without loss we focus on $U_0$, so $u=u(x_1,\ldots x_m)$ on $U_0$. We compute on the smooth locus of $u$ (which has full measure by Lemma \ref{uregularity}):
\[
\begin{split}
& \omega_t= \omega_M+ dd^c \phi_t
\\
=& \omega_M+ \frac{1}{|\log |t||}\sum_{1\leq i,j\leq m} \frac{\partial^2 u}{\partial x_i\partial x_j} d\log |F_i|_{h^{\otimes d_i}} \wedge d^c \log |F_j|_{h^{\otimes d_j} } 
\\
& -\sum_1^m \frac{\partial u}{\partial x_i} dd^c \log |F_i|_{ h^{\otimes d_i} }.
\end{split}
\]
By the Poincar\'e-Lelong equation, since we are computing away from the divisor $F_i=0$, the term
\[
 -dd^c \log |F_i|_{ h^{\otimes d_i} }= -dd^c (|F_i|e^{-d_i\phi_h})  = d_i dd^c \phi_h= d_i\omega_M.
\]
Thus $\omega_t$ simplifies into
\[
(1+ \sum_1^m d_i\frac{\partial u}{\partial x_i}  )\omega_M+  \frac{1}{|\log |t||}\sum_{1\leq i,j\leq m} \frac{\partial^2 u}{\partial x_i\partial x_j} d\log |F_i|_{h^{\otimes d_i}} \wedge d^c \log |F_j|_{h^{\otimes d_j} } .
\]
After deleting a subset of $\Delta_k$ of Lebesgue measure $\ll \delta$, using the regularity of $u$ from Lemma \ref{uregularity}, we can ensure that $u$ has $C^2$ bounds depending on $\delta$ but not on $t$, whence $1+ \sum_1^m d_i\frac{\partial u}{\partial x_i} $ has a lower bound depending on $\delta$. For $t$ sufficiently small depending on $\delta$, the distinction between $d\log |F_i|$ and $d\log |F_i|_{h^{\otimes d_i} }$ is then suppressed by the $\frac{1}{|\log |t||}$ factor, so $\omega_t$ admits an approximation
\[
(1+ \sum_1^m d_i\frac{\partial u}{\partial x_i}  )\omega_M+  \frac{1}{|\log |t||}\sum_{1\leq i,j\leq m} \frac{\partial^2 u}{\partial x_i\partial x_j} d\log |F_i| \wedge d^c \log |F_j| 
\]
with error $\ll \delta$.

After performing the binomial expansion for $\omega_t^n$, the main term is 
\begin{equation}
\frac{n!}{m!(n-m)!} (1+ \sum_1^m d_i\frac{\partial u}{\partial x_i}  )^{n-m}\omega_M^{n-m} \wedge m!\det(D^2u)_{m\times m} \bigwedge_1^m \frac{1}{|\log |t||} d\log |F_i| \wedge d^c \log |F_i| .
\end{equation}
The other terms in the binomial expansion would have to involve at least $n-m+1$ wedge products of $\omega_M$, but in the normal direction to $E_J$, the $\omega_M$ is suppressed by $\sum_{1\leq i,j\leq m} \frac{\partial^2 u}{\partial x_i\partial x_j} d\log |F_i| \wedge d^c \log |F_j| $, so such contributions can be neglected.

We now simplify this main term. Applying the real MA type equation (\ref{realMAtypeeqn}), this becomes
\[
\frac{\sum d_i}{\prod d_i}\omega_M^{n-m} \wedge m!\bigwedge_1^m \frac{1}{|\log |t||} d\log |F_i| \wedge d^c \log |F_i|.
\]
Using again the suppression of $\omega_M$ in the normal direction, and the CY condition on $\omega_{E_J}$ from (\ref{omegaEJ}), the above expression is approximated by
\[
(\sum d_i)(L^{n+1}\cdot M) C_0 \Omega_{E_J}\wedge \overline{\Omega}_{E_J} \wedge \bigwedge_1^m \frac{1}{|\log |t||} d\log |F_i| \wedge d^c \log |F_i|.
\]
By the asymptote of the normalised CY volume measure (\ref{normalisedCYmeasure}), this expression is approximated by
$
(\sum d_i)(L^{n+1}\cdot M) \mu_t.
$
We recall $-K_M=(d_0+\ldots +d_m)L$, so
\[
(L^n)=L^n\cdot X_t= L^n\cdot (-K_M)=(\sum d_i)(L^{n+1}\cdot M),
\]
and the above simplifies to $(L^n) \mu_t$.

In summary, the deleted region has $\mu_t$-measure $<\delta$, and on the good region all the approximation errors are $\ll \delta$ for sufficiently small $t$ depending on $\delta$, whence the claim.
\end{proof}

\begin{cor}\label{totalvariation2}
The total variation
\[
\int_{X_t} |\omega_t^n- (L^n)\mu_t|\to 0,\quad t\to 0.
\]
\end{cor}

\begin{proof}
Observe  $\int_{X_t}\omega_t^n= \int_X (L^n)\mu_t=(L^n)$, and apply Prop. \ref{totalvariation1}.
\end{proof}

\begin{rmk}\label{C0regularisation}
The potential $\phi_t$ is globally non-smooth. However, by an arbitrarily $C^0$-small regularisation of $\phi_t$, we can obtain a genuine K\"ahler potential, which is $C^\infty_{loc}$ close to $\phi_t$ on the smooth region, and in particular its volume measure is arbitrarily close to $\omega_t^n$ in the total variation norm.
\end{rmk}

\begin{cor}\label{hybridlimit2}
The complex MA measure $\omega_t^n$ on $X_t$ converges to the NA CY measure $(L^n)\mu_0$ as $t\to 0$. In particular, the NA metric $\phi^{NA}$ is the unique NA CY metric.
\end{cor}

\begin{proof}
For the first claim, recall that $\mu_t\to \mu_0$ in the hybrid topology (\cf Thm. \ref{Volumeconvergence}), and that $\int_{X_t} |\omega_t^n-(L^n)\mu_t|\to 0$ as $t\to 0$.  For the second claim, we apply Cor. \ref{hybridlimit1}.
\end{proof}

\section{Convergence of the CY potential}\label{ConvergenceCYpotential}

\subsection{Pluripotential theory backgrounds}

We will recall a few pluripotential theoretic estimates from  \cite{LiuniformSkoda}\cite{LiFermat}\cite{LiNA}, which build upon Kolodziej's technique of capacity estimates \cite{EGZ}\cite{Kolodziej1}\cite{Kolodziej2}.

\begin{thm}(Uniform Skoda estimate)\label{uniformSkoda}\cite{LiuniformSkoda}
 There are uniform positive constants $\alpha, A$ independent of $t$ for $0<|t|\ll 1$, such that 
	\[
	\int_{X_t} e^{-\alpha v}  d\mu_t \leq A, \quad \forall v\in PSH(X_t,\frac{1}{|\log |t||} \omega_M|_{X_t}) \text{ with } \sup_{X_t} v=0.
	\]

\end{thm}

We write the CY metric on $(X_t, L)$ with respect to the ansatz metric $\omega_t$ (\cf section \ref{Ansatzmetrics}) as \begin{equation}
\omega_{CY,t}=\omega_t+ |\log |t||dd^c \psi_t.
\end{equation}
Since the K\"ahler potential between $\omega_t$ and $\omega_M$ is bounded by $O(|\log |t||)$, in the $L^\infty$ estimate below these two background metrics may be interchanged.

\begin{thm}\label{uniformLinfty}
(Uniform $L^\infty$ bound) \cite{LiuniformSkoda} For $0<|t|\ll 1$, for some appropriate additive normalisation, the potential satisfies $\norm{\psi_t}_{L^\infty}\leq C$ independent of $t$.
\end{thm}

We will henceforth fix the additive normalisation so that $\min_{X_t} \psi_t=0$.

We record a general statement concerning K\"ahler metrics whose volume measures are close in total variation norm.

\begin{thm}\label{UniformL1stabilitythm}
	(Uniform $L^1$-stability, rephrasing \cite[Thm 2.6]{LiNA}) Let $(Y,\omega)$ be a compact K\"ahler manifold, and $\phi\in PSH(Y,\omega)\cap C^\infty$.
	Assume
	\begin{itemize}
		\item There is a Skoda estimate
		\[
		\int_Y e^{-\alpha v}  \frac{\omega^n}{\text{Vol}(Y)} \leq A, \quad \forall v\in PSH(Y,\omega) \text{ with } \sup_Y v=0.
		\]

		\item There is a mass lower bound $ \frac{1}{\text{Vol(Y)}}\int_{  \{\phi\leq a\} } \omega^n \geq b>0. $
		
		\item The total variation $ \frac{1}{\text{Vol}(Y)} \int_Y |\omega^n-\omega_\phi^n |\leq \gamma^{2n+3}<1$. 
		
		\item  $\norm{\phi}_{C^0}\leq A'$.
	\end{itemize}
	Then for $0<\gamma<\gamma_0(b, n, \alpha, A, A')\ll 1$, there is a uniform estimate 
	\[
	\sup_Y \phi \leq a+ C(b, n, \alpha, A, A') \gamma.
	\]
\end{thm}

\subsection{$C^0$-convergence in the generic region}

We first apply the uniform $L^1$-stability theorem \ref{UniformL1stabilitythm} to our main setting, to obtain a weak $L^1$-estimate on the CY potential.

\begin{cor}(Weak $L^1$-convergence estimate)\label{WeakL1convergence}
	Given any small $\delta>0$, then for sufficiently small $t$ depending on $\delta$, we have 
	\[
	\int_{ X_t\cap \{ \psi_t\leq \delta \} } d\mu_t \geq 1-\delta.
	\]
\end{cor}

\begin{proof}
	Suppose the contrary $
	\int_{ X_t\cap \{ \psi\leq \delta \} } d\mu_t < 1-\delta.
	$
	We apply the uniform $L^1$-stability theorem, with $Y=X_t$ and
	\[
	\omega=\frac{1}{|\log |t||}\omega_{CY,t},\quad \omega_\phi= \frac{1}{|\log |t||}\omega_t, \quad \phi=-\psi_t.
	\]
	(Strictly speaking, we are using a $C^\infty$-regularised version of $\omega_t$ instead of $\omega_t$, with arbitrarily $C^0$-close potentials, see Remark \ref{C0regularisation}. We shall suppress this regularisation procedure to make the arguments more transparent.)

	Thus $\frac{\omega^n}{\text{Vol}(Y)}=\mu_t $ is the normalised CY measure, which satisfies the uniform Skoda estimate by Thm \ref{uniformSkoda}, \ref{uniformLinfty}. 
	The $\norm{\phi}_{C^0}$ bound is Thm \ref{uniformLinfty}.
	The mass lower bound  with $a=-\delta$ and $b=\delta$ holds by the contradiction hypothesis. We choose $\gamma \ll C(b,n,\alpha,A,A')\delta$. The total variation norm can be made to be $\ll \gamma^{2n+3}$ by Cor. \ref{totalvariation2}, for sufficiently small $t$.
	The uniform $L^1$-stability Theorem then gives
	\[
	\sup_{X_t} (-\psi_t) \leq -\delta+ C(b,n,\alpha,A,A')\gamma \leq -\frac{\delta}{2},
	\]
	which contradicts our normalisation that $\min_{X_t} \psi_t=0$.
\end{proof}

We shall improve the weak $L^1$-convergence estimate to a $C^0_{loc}$-convergence on a 
concrete subset  lying close to $E_J$ inside $X_t\subset M$, with a nontrivial proportion of the CY measure.

We pick a chart on $E_J$ with complex coordinates $z_{m+1},\ldots z_n$, bounded away from $\{ F=0 \}$, and pick a local trivialising section of $L$, so that $F_0,\ldots, F_m$ can be identified with coordinates $z_0,\ldots z_m\in \C^*$. The equation $F_0\ldots F_m=-tF$ just means $z_0$ can be solved in terms of the other local coordinates, so the local picture of $X_t$ is biholomorphic to a subset of $B(2)^{n-m}\times (\C^*)^m$. We will focus on the region inside $B(1)^{n-m}\times (\C^*)^m$,
 \begin{equation}\label{C0locsubset}
 \frac{-\log |z_i|}{ |\log |t||  }\gtrsim \epsilon , \quad i=0, 1,\ldots m. 
 \end{equation}

\begin{rmk}
	The notation $ a\gtrsim b$ means $a\geq C^{-1}b$ with $C$ independent of $t$. This is intended to suppress some unpleasant constants arising from shrinking domains slightly.
\end{rmk}

\begin{prop}\label{C0locconvergence}
($C^0_{loc}$-potential convergence estimate)
Given any small $\epsilon$, then for sufficiently small $t$ depending on $\epsilon$, we have $\psi_t\leq C\epsilon$ on the subset (\ref{C0locsubset}).
\end{prop}

\begin{proof}
We will work with local K\"ahler potentials in the coordinate charts, and exploit the psh property through the mean value inequality.

Using the local trivialisation of $L$, we can write $h= e^{-2\phi_h}$, so the CY metric $\omega_{CY,t}$ on $X_t$ has local absolute potential
\[
\varphi= \phi_h+ \phi_t+ |\log |t||\psi_t,\quad  |\phi_h|\leq C.
\]
Let $\bar{\varphi}$ be the average of $\varphi$ over the $z_{m+1},\ldots z_n$ variables, and over the angles in the $T^m$ directions:
\[
\bar{\varphi}= \dashint_{B(2)^{n-m}}\dashint_{T^m} \varphi(z_1e^{\sqrt{-1}\theta_1},\ldots z_me^{\sqrt{-1}\theta_m}, z_{m+1},\ldots z_n). 
\]
 Thus $\bar{\varphi}$ is a function of $\log |z_1|,\ldots \log |z_m|$, and it must be convex because averaging preserves psh property.

Around a given value
\[
y=(y_0,\ldots y_m),\quad \sum_0^m y_i=1,\quad  y_i\gtrsim \epsilon, 
\]
we consider the annular subregion of (\ref{C0locsubset}) with
\[
| \frac{-\log |z_i| }{|\log |t||} -y_i|\lesssim \epsilon,\quad i=0,\ldots m.
\]
Within this region, recalling that $\phi_t=|\log |t||u(x_0,\ldots x_m)$ with $x_i=-\frac{\log |F_i|_{h^{\otimes d_i}}}{|\log |t||}$, the Lipschitz bound on $u$ implies the small oscillation estimate
\begin{equation}\label{smalloscillation}
|\frac{1}{|\log |t||} \phi_t-  u(y  )|\lesssim \epsilon.
\end{equation}
In our normalisation $\psi_t\geq 0$, so \[
\frac{1}{|\log |t||}\varphi= \frac{1}{|\log |t||}\phi_h+\frac{1}{|\log |t||} \phi_t+\psi_t \geq -C\epsilon.
\]

We now apply Cor. \ref{WeakL1convergence} with $\delta \leq \epsilon^{m+1}$, so that
$
\psi_t\leq \epsilon 
$
except on a subset with normalised CY measure $\leq \epsilon^{m+1}$.
Adding up the three terms in $\varphi$, we see that for sufficiently small $t$ depending on $\epsilon$,
\[
\begin{split}
& \int_{ | \frac{-\log |z_i| }{|\log |t||}- y_i|\lesssim \epsilon   }  |\frac{1}{|\log |t||}\varphi- u( y )| d\mu_t 
\\
& \lesssim \epsilon  \int_{ | \frac{-\log |z_i| }{|\log |t||} -y_i|\lesssim \epsilon   } d\mu_t+  \epsilon^{m+1} \norm{\psi_t}_{L^\infty} \lesssim \epsilon^{m+1}.
\end{split}
\]
Upon averaging, we conclude the local average $L^1$ upper bound on $\bar{\varphi}$,
\[
\dashint_{ | \frac{-\log |z_i| }{|\log |t||} -y_i|\lesssim \epsilon   } |\frac{1}{|\log |t||}\bar{\varphi} - u (y   ) | \lesssim \epsilon.
\]
Since $\bar{\phi}$ is a convex function, the average $L^1$-upper bound implies an $L^\infty$ upper bound after slightly shrinking the domain, namely 
\[
\frac{1}{|\log |t||}\bar{\varphi} - u (y   ) \lesssim \epsilon.
\]
But $\bar{\varphi}$ is itself obtained by averaging the psh function $\varphi$, so by the mean value inequality, and by the almost non-negativity $\frac{1}{|\log |t||}\varphi\geq -C\epsilon$, we deduce that on a slightly shrunk domain,
\[
\frac{1}{|\log |t||}\varphi - u (y ) \lesssim \epsilon, \quad \text{for } | \frac{-\log |z_i| }{|\log |t||} -y_i|\lesssim \epsilon.
\]

Unravelling the definition of $\varphi$,  we obtain
\[
\begin{split}
&\psi_t \leq  \frac{1}{|\log |t||}\varphi-  \frac{1}{|\log |t||}\phi_t-  \frac{1}{|\log |t||}\phi_h
\\
& \leq u(y) - u(y) + \frac{C}{|\log |t||}+ C\epsilon \leq C\epsilon.
\end{split}
\]
Since the region (\ref{C0locsubset}) is covered by these small charts, the claim follows.
\end{proof}

\subsection{Bergman kernel estimates}

Our goal is to improve the potential convergence from the generic region to the whole of $X_t$. The difficulty is that the complex structure is quite degenerate on the nongeneric regions of $X_t$, and the ansatz potential $\phi_t$ oscillates quite drastically, so that a simple application of the mean value inequality does not seem to suffice. Instead, we will proceed via Bergman kernel estimates.

\begin{prop}\label{OhsawaTakegoshi}
For any sufficiently large $l$, and
for any $z\in X_t$, there is a section $s\in H^0(X_t, lL)$ with $|s|_{h^{\otimes l}}(z)=1$, and
\[
\frac{1}{2l}\log \int_{X_t} |s|^2_{h^{\otimes l }} e^{ -2l (\phi_t+ |\log |t||\psi_t)  } \sqrt{-1}^{n^2}\Omega_t \wedge \overline{\Omega}_t \leq  -(\phi_t+ |\log |t||\psi_t)   (z) + \frac{C}{l} |\log |t||.
\]	
All the constants are uniform for small $t$.
\end{prop}

\begin{proof}
The idea is to apply the Ohsawa-Takegoshi extension theorem to produce some section of $lL+ K_{X_t}$ with norm control, by extending from the given point $z\in X_t$ to $X_t$. We select the positive metric on the line bundle $lL$ as a weighted combination of the CY metric $h e^{ -2(\phi_t+ |\log |t||\psi_t) } $ and the background metric $h$,
\[
h^{\otimes N_0} (h e^{ -2(\phi_t+ |\log |t||\psi_t) } ) ^{\otimes (l-N_0) }= h^{\otimes l}e^{ -2(l-N_0) (\phi_t+ |\log |t||\psi_t) },
\]
where $N_0$ is a large fixed integer independent of $l$, and $\tau_0,\ldots \tau_{N_1}\in H^0(M, N_0 L)$ induces a projective embedding of $M$, with $|\tau_i|_{h^{N_0}}\leq 1$. By taking linear combinations of constant order we can arrange $\tau=(\tau_1, \ldots \tau_{N_1})$ to have common zero precisely at $z\in X_t$, and the magnitudes are globally bounded $|\tau_i|_{h^{N_0}}\leq C$, and there is a quantitative transversality condition at $z\in X_t$,
\[
| \frac{ \Lambda^n (d\tau) }{  \Omega_t} |_{h^{\otimes nN_0}}(z) \geq  e^{-C |\log |t||}.
\] 
Here $\Lambda^nd\tau \in \Lambda^n \C^{N_1} \otimes (L^{\otimes n N_0}\otimes K_{X_t})|_z$, 
and 
the quantitive transversality amounts to estimating the ratio between the Fubini-Study volume form on $X_t$ with the CY volume form $\sqrt{-1}^{n^2}\Omega_t\wedge \overline{\Omega}_t$, which is a local computation.

The Ohsawa-Takegoshi extension theorem then produces a section $s\in H^0(X_t, lL)$ (equivalently a section $s\otimes \Omega_t\in H^0(X_t, lL+K_{X_t})$),
\[
\begin{split}
& \int_{X_t} |s|^2_{h^{\otimes l}} e^{ -2(l-N_0) (\phi_t+ |\log |t||\psi_t)  } \sqrt{-1}^{n^2}\Omega_t \wedge \overline{\Omega}_t \\
& \leq C  |s|^2_{h^{\otimes l}}(z) e^{- 2(l-N_0) (\phi_t+ |\log |t||\psi_t)(z)  } | \frac{ \Lambda^n (d\tau) }{  \Omega_t} |^{-2}_{h^{ \otimes nN_0 }}(z)
\\
& \leq  |s|^2_{h^{\otimes l}}(z) e^{- 2(l-N_0) (\phi_t+ |\log |t||\psi_t)(z)  } e^{C|\log |t||}.
\end{split}
\]
Now by the $L^\infty$ bound on the CY potential (\cf Thm. \ref{uniformLinfty}), 
\[
|\phi_t+ |\log |t||\psi_t |\leq C|\log |t||, 
\]
hence 
\[
\int_{X_t} |s|^2_{h^{\otimes l}} e^{ -2l (\phi_t+ |\log |t||\psi_t)  } \sqrt{-1}^{n^2}\Omega_t \wedge \overline{\Omega}_t \leq  (|s|^2_{h^{\otimes l}} e^{- 2l (\phi_t+ |\log |t||\psi_t)  }) (z) e^{C|\log |t||}
\]
and taking logarithm proves the claim.
\end{proof}

We will try to replace the section $s$ by monomial sections, by imitating the non-archimedean case (\cf section \ref{Decompositionofsections}). Recall that we have chosen the subspaces $V_k\subset H^0(M, kL)$ mapping isomorphically to $H^0(E_J, kL)$ upon restriction. For sufficiently small $t$, we can decompose any section $s\in H^0(X_t, lL)$ uniquely as
\begin{equation}\label{decompositions}
s= \sum F_0^{l_0}\ldots F_m^{l_m} s_{l_0,\ldots l_m},
\end{equation}
where $l_0,\ldots l_m$ are non-negative integers with at least one zero entry, $d_0l_0+\ldots d_ml_m\leq l$, and $s_{l_0,\ldots l_m}\in V_{l-\sum d_kl_k}$.

The archimedean counterpart of the ultrametric inequality is the following triangle inequality:

\begin{lem}\label{triangle}
There is a pointwise upper bound uniform for small $t$:
\[
|s|_{h^{\otimes l}}\leq C(l) \max_{l_0,\ldots l_m} |F_0|_{h^{\otimes d_0}}^{l_0}\ldots |F_m|_{h^{\otimes d_m}}^{l_m} \norm{s_{l_0\ldots l_m} }_{V_{ l-\sum d_il_i }} ,
\]
where $\norm{\cdot}_{ V_k  }$ is a fixed norm on the finite dimensional space $V_k$ independent of $t$.
\end{lem}

The archimedean counterpart of the `no cancellation property' (\cf Cor. \ref{logsonessentialskeleton}) comes from a more subtle \emph{almost orthogonality} property. We work in the same coordinate setup $B(1)^{n-m}\times (\C^*)^m$ as Prop. \ref{C0locconvergence}. Around any given value
\begin{equation}\label{domainy}
y=(y_0,\ldots y_m), \quad \sum_0^m y_i=1,\quad y_i\gtrsim \epsilon, 
\end{equation}
 we consider the small annulus subregion of (\ref{C0locsubset}),
\[
U_{y,t}=\{ y_i |\log |t|| \leq -\log |z_i| \leq y_i |\log |t||+1, \forall 1\leq i\leq m \} \subset B(1)^{n-m}\times (\C^*)^m,
\]
and define the inner product on $H^0(X_t, kL)$
\[
\norm{s}_{y,t}^2=\int_{U_{y,t}} |s|_{h^{\otimes l}}^2  \sqrt{-1}^{n^2} \Omega_t\wedge \overline{\Omega}_t.
\]
We can arrange $U_{y,t}$ to be symmetric under the $T^m$ action rotating $z_0,\ldots z_m$ but fixing their product.

\begin{lem}\label{almostorthogonality1}
(Almost orthogonality I)
Let $t$ be sufficiently small depending on $\epsilon, l$. 
Suppose $s= F_0^{l_0}\ldots F_m^{l_m} \sigma$ with $\sigma \in V_{l-\sum d_kl_k}$, where $l_0,\ldots l_m$ are non-negative integers with at least one zero entry, $d_0l_0+\ldots d_ml_m\leq l$, and $\sigma\in V_{l-\sum d_kl_k}$. Then
\[
C(l)^{-1} \norm{\sigma}_{V_{l-\sum d_kl_k  }} \leq |t|^{ -\sum_0^m l_iy_i}\norm{s}_{  y,t }\leq C(l)  \norm{\sigma}_{V_{l-\sum d_kl_k  }}.
\]
\end{lem}

\begin{proof}
In coordinates, we can regard sections as local holomorphic functions, and the integral norm amounts to
\[
\norm{s}_{y,t}^2=\int_{U_{y,t}} |\sigma|^2 |z_0|^{2l_0} \ldots |z_m|^{2l_m} e^{-2l\phi_h} \sqrt{-1}^{n^2} \Omega_t\wedge \overline{\Omega}_t.
\]
Notice the region $U_{y,t}$ is within  $\omega_M$-distance $O(|t|^{\epsilon})$ to $E_J\subset M $, and since $\sigma$ does not vanish identically on $E_J$ by the definition of $V_{l-\sum d_il_i}$, up to a small relative error of order $O(|t|^\epsilon)$ we can ignore the dependence of $\sigma$ and $\phi_h$ on the $z_0,\ldots z_m$ variables, and regard them as functions of $z_{m+1},\ldots z_n$ only. The integral then essentially splits as a product. The $\sigma$ part contributes a factor $\int_{B(1)} e^{-2l\phi_h}|\sigma|^2 \Omega_{E_J}\wedge \overline{\Omega}_{E_J}$, which is $C(l)$-uniformly equivalently to $\norm{\sigma}_{V_{l-\sum d_il_i}}^2$. The other part contributes another factor proportional to
\[
\int |z_0|^{2l_0}\ldots |z_m|^{2l_m} \bigwedge_1^m\sqrt{-1}d\log z_i\wedge d\overline{\log z_i}.
\]
Since the integration takes place within one dyadic scale, this integral is uniformly equivalent to $|t|^{2l_0y_0}\ldots |t|^{2l_my_m}$. 
\end{proof}

\begin{lem}
(Almost orthogonality II) \label{almostorthogonality2} 
Let $t$ be sufficiently small depending on $\epsilon, l$.
Let $s=  F_0^{l_0}\ldots F_m^{l_m} \sigma$  and $s'= F_0^{l_0'}\ldots F_m^{l_m'} \sigma'$ be as in Lemma \ref{almostorthogonality1}. Then the  inner product
\[
|\langle |t|^{-\sum l_iy_i} s, |t|^{-\sum l_i'y_i} s'\rangle_{y,t} |\leq C(l)|t|^\epsilon \norm{\sigma}\norm{\sigma'} ,
\]
unless $(l_0,\ldots l_m)=(l_0',\ldots l_m')$.
\end{lem}

\begin{proof}
As in Lemma \ref{almostorthogonality1}, the inner product will essentially split into a product, with an error controlled by $C(l)|t|^\epsilon \norm{\sigma}\norm{\sigma'}$. The interesting factor for us is
\[
\int (\frac{1}{z_1\ldots z_m})^{l_0} z_1^{l_1}\ldots z_m^{l_m}  \overline{(\frac{1}{z_1\ldots z_m})^{l_0'} z_1^{l_1'}\ldots z_m^{l_m'} } \bigwedge_1^m\sqrt{-1}d\log z_i\wedge d\overline{\log z_i}.
\]
By Fourier orthogonality, this is only nonzero when
\[
l_1-l_0= l_1'-l_0',\ldots, l_m-l_0=l_m'-l_0',
\]
namely $(l_0,\ldots l_m)= (l_0',\ldots l_m')$ modulo $\R(1,1\ldots 1)$. But $\min_i l_i=\min_i l_i'=0$ by assumption, so  $(l_0,\ldots l_m)= (l_0',\ldots l_m')$.
\end{proof}

\begin{cor}
(No cancellation)
\label{almostorthogonality3}  
Let $t$ be sufficiently small depending on $\epsilon, l$. Given a decomposition of the section $s$ as in (\ref{decompositions}),
then 
\[
\max_{l_0,\ldots l_m}  \norm{s_{l_0,\ldots l_m}}_{V_{l-\sum l_id_i}} |t|^{\sum l_i y_i}  \leq C(l)\norm{s}_{y,t}.
\]
\end{cor}

\begin{proof}
By the almost orthogonality property,
\[
\begin{split}
\norm{s}_{y,t}^2\geq & (1-C(l)|t|^\epsilon) \sum \norm{ F_0^{l_0}\ldots F_m^{l_m} s_{l_0,\ldots l_m} }_{ y,t  }^2\\
\geq & C(l)^{-1} \sum \norm{s_{l_0,\ldots l_m}}_{V_{l-\sum d_il_i}}^2 |t|^{2\sum l_iy_i }. 
\end{split}
\]
\end{proof}

The following outcome will be a key ingredient in the proof of the main Theorem \ref{mainthm}.

\begin{prop}\label{Bergmankernelestimate}
	Given any small $\epsilon$ and large $l$, the following holds for sufficiently small $t$ depending on $\epsilon, l$.
	For any given $z\in X_t$, there exists $(l_0,\ldots l_m)$, such that for any
    $y=(y_0, \ldots y_m)$ in the domain (\ref{domainy}), there is an upper bound on $\psi_t$ at $z\in X_t$,
    	\begin{equation}\label{globalupperbound}
    	\psi_t(z) \leq \frac{-1}{|\log |t||}\phi_t(z)+u(y) +\sum \frac{l_i}{l} \frac{ \log |F_i|_{h^{\otimes d_i}}(z)}{|\log |t||} +  \frac{\sum l_iy_i}{l} + ( \frac{C}{l}+ C \epsilon ).
    	\end{equation}
	The constants $C$ are independent of $t, z, \epsilon, l, y$. 
\end{prop}

\begin{proof}
	Given any $z\in X_t$, 
	we consider the section $s\in H^0(X_t, lL)$ constructed via Ohsawa-Takegoshi in Prop. \ref{OhsawaTakegoshi}, and decompose it according to (\ref{decompositions}). Within the small annulus region $U_{y,t}$,  the small oscillation estimate
	\[
	| \frac{1}{|\log |t||}\phi_t- u( y_0,\ldots y_m ) | \lesssim \epsilon,
	\]
	holds as in (\ref{smalloscillation}),
	and by Prop. \ref{C0locconvergence} we have $\psi_t\leq C\epsilon$ on $U_{y,t}$. Thus
	\[
	\begin{split}
	&  \int_{X_t} |s|^2_{h^{\otimes l}} e^{ -2l (\phi_t+ |\log |t||\psi_t)  } \sqrt{-1}^{n^2}\Omega_t \wedge \overline{\Omega}_t 
	\\
	& \geq e^{ -Cl|\log |t||\epsilon -2l|\log |t||u(y) } \int_{X_t} |s|^2_{h^{\otimes l}}  \sqrt{-1}^{n^2}\Omega_t \wedge \overline{\Omega}_t 
	\\
	& \geq e^{ -Cl|\log |t||\epsilon -2l|\log |t||u(y) } \int_{U_{y,t}} |s|^2_{h^{\otimes l}}  \sqrt{-1}^{n^2}\Omega_t \wedge \overline{\Omega}_t .
	\end{split}
	\]
	Taking logarithm and rearranging, we obtain
	\[
	\begin{split}
	\frac{1}{l}\log \norm{s}_{y,t} \leq &  \frac{1}{2l}\log  \int_{X_t} |s|^2_{h^{\otimes l}} e^{ -2l (\phi_t+ |\log |t||\psi_t)  } \sqrt{-1}^{n^2}\Omega_t \wedge \overline{\Omega}_t + |\log |t||u(y)+ C |\log |t||\epsilon.
	\\
	\leq & -(\phi_t+|\log |t||\psi_t)(z)+|\log |t||u(y) + \frac{C}{l}|\log |t||+C |\log |t||\epsilon.
	\end{split}
	\]
	Applying the no cancellation property in Cor. \ref{almostorthogonality3}, 
	\[
	\begin{split}
	& \max_{l_0,\ldots l_m} \frac{1}{l} \log \norm{s_{l_0,\ldots l_m} }_{V_{l-\sum d_il_i}} + \frac{\sum l_iy_i}{l}\log |t| \leq \frac{1}{l}\log \norm{s}_{y,t} + C(l) \\
	\leq &  -(\phi_t+|\log |t||\psi_t)(z)+|\log |t||u(y) + \frac{C}{l}|\log |t||+ C |\log |t||\epsilon +C(l).
	\end{split}
	\]

	On the other hand, by construction $|s|_{h^{\otimes l}}(z)=1$, so the triangle inequality of Lemma \ref{triangle} implies
	\[
	\max_{l_0,\ldots l_m} \sum_0^m l_i \log |F_i|_{h^{\otimes d_i}}(z) +\log \norm{s_{l_0,\ldots l_m}}_{V_{l-\sum d_il_i}} \geq -C(l). 
	\]
	Contrasting the above two inequalities, we can find some $l_i$ depending on $z$ but not on $y$, such that
	\[
	\begin{split}
	& \sum_0^m \frac{l_i}{l} \log |F_i|_{h^{\otimes d_i}}(z)  \geq -C(l)- \frac{1}{l} \log \norm{s_{l_0,\ldots l_m}}_{V_{l-\sum d_il_i}} \\
	& \geq  (\phi_t+|\log |t||\psi_t)(z)-|\log |t||u(y) + \frac{\sum l_iy_i}{l}\log |t|- ( \frac{C}{l}|\log |t||+ C |\log |t||\epsilon +C(l) ).
	\end{split}
	\]
	After rearranging,
	\[
	\psi_t(z) \leq \frac{-1}{|\log |t||}\phi_t(z)+u(y) +\sum \frac{l_i}{l} \frac{ \log |F_i|_{h^{\otimes d_i}}(z)}{|\log |t||} +  \frac{\sum l_iy_i}{l} + ( \frac{C}{l}+ C \epsilon + \frac{C(l)}{|\log |t||}  ).
	\]
	The term $\frac{C(l)}{|\log |t||} $ can be absorbed into $Cl^{-1}$ for small enough $t$. 
\end{proof}

\subsection{Global $C^0$-convergence}

Let us take stalk of our progress so far. By Cor. \ref{hybridlimit1}, \ref{hybridlimit2}, the NA CY potential $\phi_{CY,0}$ is the $C^0$-hybrid topology limit of $\frac{1}{|\log |t||}\phi_t$. Since the CY metric on $L\to X_t$ is defined by the Hermitian metric
\[
h_{CY,t}=he^{-2(\phi_t+  |\log |t||\psi_t ) },
\] 
Theorem \ref{mainthm} reduces to showing

\begin{prop}
$\norm{\psi_t}_{C^0}\to 0$ as $t\to 0$.
\end{prop}

\begin{proof}
By construction $p=(-\frac{l_0}{l},\ldots -\frac{l_m}{l}  )\in \Delta^\vee$. 
Since $u$ is Lipschitz on $\Delta$, and the domain (\ref{domainy}) is obtained by $\epsilon$-shrinking of $\Delta$, we see that for any $p\in \Delta^\vee$,
\[
\sup_{y} \langle y, p\rangle -u(y) \geq u^*(p)-C\epsilon. 
\]
We start with the upper bound (\ref{globalupperbound}) and optimize among $y$ in the domain (\ref{domainy}), to deduce
\[
\psi_t(z) \leq \frac{-1}{|\log |t||}\phi_t(z)-u^*(-\frac{l_i}{l}) +\sum \frac{l_i}{l} \frac{ \log |F_i|_{h^{\otimes d_i}}(z)}{|\log |t||} + ( \frac{C}{l}+ C \epsilon ).
\]
Writing $x_i= - \frac{ \log |F_i|_{h^{\otimes d_i}}(z)}{|\log |t||}$, then the above becomes
\[
\psi_t(z) \leq -u(x)-u^*(-\frac{l_i}{l}) +\sum_0^m \frac{-l_i}{l} x_i + ( \frac{C}{l}+ C \epsilon ).
\]
By construction
\[
u(x)= \max_{p\in \Delta^\vee} \langle p, x\rangle-u^*(p) \geq -u^*(-\frac{l_i}{l}) +\sum_0^m \frac{-l_i}{l} x_i,
\]
so $0\leq \psi_t(z) \leq \frac{C}{l}+ C \epsilon$, when $t$ is sufficiently small depending on $l,\epsilon$. We choose $l$ to be comparable to $\epsilon^{-1}$. Since $\epsilon$ is arbitrarily small, we deduce $\norm{\psi_t}_{C^0}\to 0$ as $t\to 0$.
\end{proof}

\section{Further discussions}\label{Furtherdiscussions}

We discuss the relation to some recent literature, and pose some open questions.

\subsection{Regularity of optimal transport solution}

We collect some remarks on our optimal transport problem:

\begin{itemize}
    \item  In the analogous problem for the large complex structure limit $m=n$ for toric Fano hypersurfaces \cite{LiFano}\cite{Hultgren2}, it is noticed that the gradient of the solution also induces a wall-chamber structure, closely related to the problem of finding a canonical singular affine structure on the SYZ base. While the optimal regularity results in that context is still not fully understood, Andreasson-Hultgren \cite[section 8]{Hultgren2} have raised some challenges to the expectation of codimension two singularity in the affine structure, and suggested instead that the singular set may be bigger in some cases.

    \item  The regularity in Lemma \ref{uregularity} is likely far from optimal. For instance, one may expect $u$ to be strictly convex in the interior of $\Delta$, based on analogy with Caffarelli \cite{Caffarelli3}, which however does not directly apply since it assumes global two-sided density bound.

\end{itemize}

\subsection{Small complex structure limit}

The recent breakthrough of Sun-Zhang \cite{SunZhang} offers a satisfactory gluing construction of the CY metric in the $m=1$ special case, namely
\[
X_t=\{ F_0F_1+tF=0  \}\subset M.
\]
 Algebro-geometrically, the family of CY hypersurfaces $X_t$ degenerate into $E_0\cup E_1$ in the central fibre. The gluing involves the following ingredients:

\begin{itemize}
\item The non-compact spaces $E_0\setminus E_1$ (resp. $E_1\setminus E_0$) can be regarded as the complement of a smooth anticanonical divisor inside a Fano manifold, so carries the \emph{Tian-Yau metric} \cite{TianYau}, whose asymptotic geometry is captured by the \emph{Calabi ansatz} on the total space of a positive line bundle over the CY manifold $E_0\cap E_1$. These Tian-Yau metrics are the rescaled local models for the CY metric on $X_t$ in the region away from $E_0\cap E_1$.

\item  The region of $X_t$ around the divisor intersection $E_0\cap E_1$ is modelled on an incomplete \emph{Ooguri-Vafa type metric}, constructed using a Gibbons-Hawking type ansatz.
This Ooguri-Vafa type metric is $S^1$-symmetric, with fixed point locus along $\{F=F_0=F_1=0 \}\subset E_0\cap E_1$. The neighbourhood of the fixed point locus looks like a fibration by very collapsed Taub-NUT metrics over $E_0\cap E_1\cap \{ F=0\}$. There are two incomplete ends of the Ooguri-Vafa type metric, which match up with the asymptotes of the two Tian-Yau regions.

\end{itemize}

The $m=1$ special case of our main result fits with this picture as follows. The solution to the optimal transport problem is explicitly computed in (\ref{optimaltransportm=1}). In terms of the potential $u$, modulo linear functions of $x$ (which give rise to almost pluriharmonic functions, so have little effect on the K\"ahler metric), on each of the two intervals $u$ is proportional to a power function with exponent $\frac{n+1}{n}$. This solution encodes the Calabi ansatz metric appearing in \cite{SunZhang}, which occupies almost the full measure of $X_t$ for $0<|t|\ll 1$. .

Our potential theoretic description does not capture the full nuance of the Tian-Yau core region, or the Ooguri-Vafa type metric. The Tian-Yau pieces occur at the two ends of the interval $[0,1]$, and the Ooguri-Vafa type metric occurs at the transition point $x=\frac{d_1}{d_0+d_1}$.
These effects happen on microscopic length scales invisible to the global limiting potential.

\subsubsection{Expected picture for $m>1$: Ooguri-Vafa type metrics}

For $m>1$, the $x= \frac{d_1}{d_0+d_1}$ transition point is supposedly replaced by the walls between the various chambers. We expect that along the walls separating two chambers $\Delta_i, \Delta_j$, we may see an Ooguri-Vafa type metric similar to the $m=1$ case, which contains a fibration by some (twisted) product of very collapsed Taub-NUT metrics with collapsed cylinders $(\C^*)^{m-1}$, over the base $\{ F_0=\ldots =F_m=F=0\}$. A new feature of $m>1$ is the triple junction between the walls; as a guess, we may see a fibration by some (twisted) product of the Taub-NUT type metric on $\C^3$ \cite{LiTaub} with collapsed cylinders $(\C^*)^{m-2}$, over the base  $\{ F_0=\ldots =F_m=F=0\}$. The deeper intersections between the walls may exhibit an inductive pattern.

\subsubsection{Generalised Tian-Yau metrics}

In the recent joint work with Collins \cite{CollinsLi} we constructed a generalisation of the Tian-Yau metric, on the complement of an anticanonical divisor $D$ inside a compact Fano manifold, where $D$ consists of two transverse smooth divisor components. In the $m=2$ case of our main setting, this construction would apply to $E_0\setminus (E_1\cup E_2)$, and cyclic permutations thereof. The asymptotic geometry is semi-explicit, involving hypergeometric functions.

The $m=2$ analogue of the two endpoints of the interval $[0,1]$, is the triangle $\partial \Delta$. It is conceivable that along the generic part of this triangle, we see a parametrised version of the usual Tian-Yau metrics on $E_0\cap E_1\setminus E_2$ (and cyclic permutations etc), but at special points on $\partial \Delta$, the generalised Tian-Yau construction may appear as the local model.

\subsection{Numerical computation for CY metrics}

There is some interest among physicists to compute CY metrics numerically in some polarised degeneration limit $t\to 0$. A traditional strategy, which keeps the K\"ahler property manifest, is to represent the CY potential via Fubini-Study approximation by a large Hermitian matrix on $H^0(X_t, kL )$ for $k\gg 1$, and then try to solve the Monge-Amp\`ere equation numerically \cite{Douglas}. However, for fixed $k$, the Fubini-Study volume form is very different from the CY volume form as $t\to 0$: in the hybrid topology, the limit of the Fubini-Study volume form is an atomic measure, while the limit of the normalised CY volume form is a Lebesgue measure. Thus to achieve reasonable precision for small $t$, one expects to take $k$ large, which would be computationally expensive.

Our result suggests a very different perspective: as long as one is only interested in the K\"ahler potential of the CY metrics up to small $C^0$-error, then it suffices to solve an optimal transport problem, which is computationally efficient because the Kontorovich formulation transforms it into a linear optimization problem with convex cost.

It would be interesting if further numerical methods can be developed to capture the microscopic metric information, such as the Tian-Yau core region or the Ooguri-Vafa type neck region.

\end{document}